\documentclass[11pt]{amsart}
\usepackage[margin=1.2in,includeheadfoot]{geometry}
\usepackage{amsmath,amssymb,amsthm,amsfonts,amscd,mathtools}
\usepackage{bm,mathrsfs,stmaryrd,dsfont,esint,latexsym}
\usepackage{tikz-cd}
\usepackage{slashed}
\usepackage{indentfirst}
\usepackage{enumerate,enumitem}
\usepackage{etoolbox,aliascnt}
\usepackage{microtype}
\usepackage{booktabs,array,longtable,tabularx,adjustbox,graphicx,float}
\usepackage{comment}
\usepackage{xcolor,xurl}
\definecolor{citeblue}{RGB}{0,0,255}
\usepackage[
 colorlinks=true,
 linkcolor=citeblue,
 urlcolor=citeblue,
 citecolor=citeblue,
 anchorcolor=black,
 linktoc=all
]{hyperref}
\usepackage[nameinlink,noabbrev,capitalise]{cleveref}
\usepackage{fancyhdr}

\numberwithin{equation}{section}
\allowdisplaybreaks[3]
\selectfont
\setlist[enumerate]{leftmargin=2.35em,itemsep=0.25em,topsep=0.35em}

\theoremstyle{plain}
\newtheorem{theorem}{Theorem}[section]

\newaliascnt{corollary}{theorem}

\aliascntresetthe{corollary}

\newaliascnt{lemma}{theorem}
\newtheorem{lemma}[lemma]{Lemma}
\aliascntresetthe{lemma}

\newaliascnt{proposition}{theorem}
\newtheorem{proposition}[proposition]{Proposition}
\aliascntresetthe{proposition}

\theoremstyle{definition}
\newaliascnt{definition}{theorem}

\aliascntresetthe{definition}

\theoremstyle{remark}
\newaliascnt{remark}{theorem}

\aliascntresetthe{remark}
\newtheorem*{remark*}{Remark}

\crefname{theorem}{Theorem}{Theorems}
\crefname{proposition}{Proposition}{Propositions}
\crefname{lemma}{Lemma}{Lemmas}
\crefname{corollary}{Corollary}{Corollaries}
\crefname{definition}{Definition}{Definitions}
\crefname{remark}{Remark}{Remarks}
\crefname{equation}{equation}{equations}
\crefname{section}{Section}{Sections}
\crefname{subsection}{Section}{Sections}
\crefname{subsubsection}{Section}{Sections}
\crefname{appendix}{Appendix}{Appendices}

\crefformat{theorem}{Theorem~#2#1#3}
\crefmultiformat{theorem}{Theorems~#2#1#3}{,~#2#1#3}{,~#2#1#3}{ and~#2#1#3}
\crefrangeformat{theorem}{Theorems~#2#1#3--#5#4#6}
\crefformat{proposition}{Proposition~#2#1#3}
\crefformat{lemma}{Lemma~#2#1#3}
\crefformat{corollary}{Corollary~#2#1#3}
\crefformat{definition}{Definition~#2#1#3}
\crefformat{remark}{Remark~#2#1#3}
\crefformat{equation}{equation~#2#1#3}          
\crefformat{section}{Section~#2#1#3}
\crefformat{subsection}{Section~#2#1#3}         
\crefformat{subsubsection}{Section~#2#1#3}       
\crefformat{appendix}{Appendix~#2#1#3}
\newcommand{\R}{\mathbb R}
\newcommand{\C}{\mathbb C}
\newcommand{\ii}{\mathrm i}
\newcommand{\dd}{\,\mathrm d}
\newcommand{\Id}{\mathrm{Id}}
\newcommand{\Span}{\operatorname{span}}
\newcommand{\Ker}{\ker}
\newcommand{\Ran}{\operatorname{ran}}
\newcommand{\dist}{\operatorname{dist}}

\newcommand{\diag}{\operatorname{diag}}

\newcommand{\cD}{\mathcal D}
\newcommand{\cH}{\mathcal H}
\newcommand{\cG}{\mathcal G}
\newcommand{\cK}{\mathcal K}
\newcommand{\cE}{\mathcal E}
\newcommand{\cF}{\mathcal F}
\newcommand{\cL}{\mathcal L}

\newcommand{\cB}{\mathcal B}
\newcommand{\cC}{\mathcal C}
\newcommand{\cP}{\mathcal P}
\newcommand{\cR}{\mathcal R}
\newcommand{\cT}{\mathcal T}
\newcommand{\cW}{\mathcal W}
\newcommand{\cO}{\mathcal O}
\newcommand{\cS}{\mathcal S}

\newcommand{\norm}[2][]{\left\lVert #2\right\rVert_{#1}}
\newcommand{\abs}[1]{\left\lvert #1\right\rvert}
\newcommand{\ipR}[2]{\left\langle #1,#2\right\rangle_{\R}}
\newcommand{\ip}[2]{\left\langle #1,#2\right\rangle_{2}}
\newcommand{\restr}[2]{\left.#1\right|_{#2}}

\begin{document}

\title[Dirac-Coulomb ground states]{Uniqueness, Nondegeneracy, and
Asymptotic Expansions of Ground States for the Dirac-Coulomb System}

\author[Chen]{Pan Chen}

\author[Guo]{Qi Guo}

\author[Zeng]{Xiaoyu Zeng}

\subjclass[2020]{35Q40; 35J50; 49J35}
\keywords{Dirac-Coulomb system; ground state; nonrelativistic limit;
nondegeneracy; uniqueness}
\date{}

\begin{abstract}
We study ground states of the following Dirac-Coulomb system 
$$\mathscr{D}_cu_c-
  \left(\abs{x}^{-1}\ast \abs{u_c}^2\right)u_c=\omega_cu_c,$$ under the constraint
  $\norm[L^2]{u_c}^2=1.$
For sufficiently large \(c\), we prove uniqueness of ground states
modulo translations, phase transformations and spinorial rotations, and determine the
kernel of the linearized operator at ground state. 
We also obtain asymptotic expansions of
ground state in the nonrelativistic limit regime.
\end{abstract}

\maketitle
\vskip0.4cm
\begingroup
\hypersetup{linkcolor=blue}
\tableofcontents
\endgroup

\section{Introduction}
{

The attractive Dirac--Coulomb system describes a spin-$1/2$ field
coupled to a Coulomb potential generated by its own density.
It arises in semiclassical models of Dirac fermions interacting
with optical phonons \cite{ComechZubkov2013}.
The fermion distorts the surrounding medium, and the induced
polarization acts back on it through an attractive potential.
This mechanism can produce a localized state consisting of the
fermion and its polarization cloud, known as a polaron.
In the semiclassical description, such bound states are represented
by localized standing waves of the coupled field equations.

Let $m>0$ denote the particle rest mass and $c>0$ the speed of light.
For convenience, let Planck constant $\hbar=1$.
The corresponding effective evolution system is
\begin{equation*}
  \left\{
  \begin{aligned}
    &\ii\partial_t u=\mathscr{D}_c u-\kappa\phi u,\\
    &-\Delta\phi=4\pi\abs{u}^2,
  \end{aligned}
  \right.
  \qquad (t,x)\in\R\times\R^3,
\end{equation*}
where $u(t,x)\in\C^4$ is the spinor field, $\phi(t,x)\in\R$
is the Coulomb field, and $\kappa>0$ is the coupling constant.
Here
\begin{equation*}
  \mathscr{D}_c:=-\ii c\,\alpha\cdot\nabla+mc^2\beta
\end{equation*}
is the free Dirac operator, with
\begin{equation*}
  \alpha_j=
  \begin{pmatrix}
    0&\sigma_j\\
    \sigma_j&0
  \end{pmatrix},
  \qquad
  \beta=
  \begin{pmatrix}
    I_2&0\\
    0&-I_2
  \end{pmatrix}.
\end{equation*}
The matrix $I_2$ is the $2\times2$ identity matrix, and
\begin{equation*}
  \sigma_1=
  \begin{pmatrix}
    0&1\\
    1&0
  \end{pmatrix},
  \qquad
  \sigma_2=
  \begin{pmatrix}
    0&-\ii\\
    \ii&0
  \end{pmatrix},
  \qquad
  \sigma_3=
  \begin{pmatrix}
    1&0\\
    0&-1
  \end{pmatrix}
\end{equation*}
are the Pauli matrices.

For simplicity, we set $\kappa=1$.
Substituting 
$\phi=\abs{x}^{-1}\ast\abs{u}^2$ and  
$u(t,x)=e^{-\ii\omega_c t}u_c(x)$, we obtain the normalized
stationary system
\begin{equation}
    \mathscr{D}_cu_c-
    (\abs{x}^{-1}\ast\abs{u_c}^2)u_c
    =\omega_cu_c,\qquad
    \norm[L^2]{u_c}^2=1.
  \label{eq:1.1}
\end{equation}
The frequency $\omega_c$ is the Lagrange multiplier associated
with the mass constraint.
Set
\begin{equation*}
  \begin{gathered}
    u_c=\binom{f_c}{g_c},
    \qquad
    f_c,g_c:\R^3\to\C^2,
    \qquad
    \cD:=-\ii\sigma\cdot\nabla,\\
    h_c:=cg_c,
    \qquad
    \lambda_c:=mc^2-\omega_c.
  \end{gathered}
\end{equation*}
Then \eqref{eq:1.1} is equivalent to
\begin{equation}
  \left\{
  \begin{aligned}
    &\cD h_c+
    \left\{
      \lambda_c-
      \abs{x}^{-1}\ast
      \left(\abs{f_c}^2+c^{-2}\abs{h_c}^2\right)
    \right\}f_c=0,\\
    &\cD f_c-2mh_c+
    c^{-2}\left\{
      \lambda_c-
      \abs{x}^{-1}\ast
      \left(\abs{f_c}^2+c^{-2}\abs{h_c}^2\right)
    \right\}h_c=0,\\
    &\norm[L^2]{f_c}^2+
    c^{-2}\norm[L^2]{h_c}^2=1.
  \end{aligned}
  \right.
  \label{eq:1.2}
\end{equation}

In the nonrelativistic Landau--Pekar theory, a polaron is described
by a Schr\"odinger particle coupled to the polarization of the
medium. Minimizing the energy with respect to the polarization
field yields an effective Coulomb self-attraction
\cite{LewinRougerie2013}.
Retaining the two spin components and neglecting spin-dependent
interactions, the corresponding normalized stationary Hartree
equation is
\begin{equation}
    -\frac1{2m}\Delta f+\lambda f-
    \left(\abs{x}^{-1}\ast\abs{f}^2\right)f=0,\qquad
    \norm[L^2]{f}^2=1.
  \label{eq:1.3}
\end{equation}
Here $f:\R^3\to\C^2$.
This is the two-component form of the Choquard--Pekar equation.

 \eqref{eq:1.1} and \eqref{eq:1.3} describe the same
effective Coulomb attraction with Dirac and Schr\"odinger
 operators, respectively.
Their mathematical connection is established by the
nonrelativistic convergence of ground states in \cite{ChenDingGuo2026}.
}

Both \eqref{eq:1.1} and \eqref{eq:1.3} have variational structure, but the corresponding
functionals have rather different geometry.  On
\begin{equation*}
 \cS:=\left\{u\in H^{1/2}(\R^3,\C^4):\norm[L^2]{u}=1\right\},
\end{equation*}
solutions of \eqref{eq:1.1} are constrained critical points of
\begin{equation}
 \cE_c(u)
 :=\int_{\R^3}u^\dagger\mathscr{D}_cu\,\dd x
 -\frac12\int _{\R^3\times\R^3}
 \frac{\abs{u(x)}^2\abs{u(y)}^2}{\abs{x-y}}\,\dd x\dd y.
 \label{eq:1.4}
\end{equation}
Since
\begin{equation*}
 \sigma(\mathscr{D}_c)=(-\infty,-mc^2]\cup[mc^2,\infty),
\end{equation*}
$\cE_c$ is strongly indefinite and unbounded from below on $\cS$.  Let
\begin{equation*}
 P_c^\pm
 :=\frac12\left(I_4\pm\frac{\mathscr{D}_c}{\abs{\mathscr{D}_c}}\right),
 \qquad
 E_c^\pm:=P_c^\pm H^{1/2}(\R^3,\C^4),
 \qquad
 u^+:=P_c^+u.
\end{equation*}
The ground state to \eqref{eq:1.1} can be obtained by considering the following
min-max problem
\begin{equation}
 e_c
 :=\inf_{\substack{w\in E_c^+\\ \norm[L^2]{w}=1}}
 \ \sup_{\substack{u\in\cS\\u^+\in\Span\{w\}}}\cE_c(u),
 \label{eq:1.5}
\end{equation}
 see \cite{Nolasco2021,ChenDingGuo2026} for more details.
A function $u_c\in\cS$ is called a ground state of \eqref{eq:1.1} if
\begin{equation*}
 \mathrm d\cE_c|_{\cS}(u_c)=0,
 \qquad
 \cE_c(u_c)=e_c.
\end{equation*}
The set of all ground states is denoted by
\begin{equation*}
 \cG_c:=\left\{u\in\cS:\cE_c(u)=e_c,\
 \mathrm d\cE_c|_{\cS}(u)=0\right\}.
\end{equation*}
It is clear that weak solutions of \eqref{eq:1.3} correspond to critical points of
the functional
\begin{equation}
  \cE_\infty(f)
  :=\frac1{2m}\norm[L^2]{\nabla f}^2
  -\frac12\int _{\R^3\times\R^3}
  \frac{\abs{f(x)}^2\abs{f(y)}^2}{\abs{x-y}}\,\dd x\dd y
  \label{eq:1.6}
\end{equation}
on the $L^2$-sphere
\begin{equation*}
  \cS':=\left\{
    f\in H^1(\R^3,\C^2):\norm[L^2]{f}=1
  \right\}.
\end{equation*}
The Gagliardo-Nirenberg
inequality implies $\cE_\infty$ is bounded from below on $\cS'$, and we
denote its infimum by
\begin{equation}
  e_\infty:=\inf_{f\in\cS'}\cE_\infty(f).
  \label{eq:1.7}
\end{equation}
A function $f\in\cS'$ is called a ground state of
\eqref{eq:1.3} if $\cE_\infty(f)=e_\infty$.
The set of all ground states of \eqref{eq:1.3} is denoted by
\[
 \cG_\infty:=\{f\in\cS':\cE_\infty(f)=e_\infty\}.
\]

The limiting minimization problem is the Choquard-Pekar problem.  Its
existence theory is a standard application of concentration compactness
\cite{Lions1984}. Lieb proved that the positive  minimizer is unique up
to translations \cite{Lieb1977}, and further qualitative properties were
established in \cite{MorozVanSchaftingen2013}.  The nondegeneracy of the
ground state follows from the analysis in
\cite{Lenzmann2009}.  We denote by $Q$ the unique positive radial 
ground state with unit mass.  It satisfies
\begin{equation}
 -\frac1{2m}\Delta Q+\lambda_\infty Q-
 (\abs{x}^{-1}\ast Q^2)Q=0,
 \qquad
 \norm[L^2]{Q}=1,
 \label{eq:1.8}
\end{equation}
for a uniquely determined $\lambda_\infty>0$.  We fix
\begin{equation*}
 \mathbf e_1=\binom10,
 \qquad
 \mathbf e_2=\binom01,
 \qquad
 f_\infty:=Q\mathbf e_1,
 \qquad
 h_\infty:=\frac1{2m}\cD f_\infty.
\end{equation*}

For the Dirac equation, 
we can not obtain the ground state by a direct minimization argument, due to
the strong indefiniteness of \eqref{eq:1.4}.  Variational
constructions of mass-constrained solitary waves for Maxwell-Dirac models
were developed in \cite{Nolasco2021}.  For the Dirac-Coulomb class
containing \eqref{eq:1.1}, the existence of ground states, their regularity,
and their nonrelativistic limit were obtained in \cite{ChenDingGuo2026}.
The properties needed below are collected in Proposition~\ref{prop:2.1}.

To the best of our knowledge, no uniqueness result is available for the
ground states of the Dirac-Coulomb system, after modulo the natural
symmetries of the equation.  Motivated by perturbative argument for the pseudo-relativistic
Hartree equation \cite{Lenzmann2009}, see also
\cite{GuoZeng2020,1}, we combine the nonrelativistic limit with the uniqueness and nondegeneracy of the ground state of the
nonrelativistic Hartree equation.  This allows us to identify the complete set
\(\cG_c\) as a single orbit under translations, phase transformations, and
spinorial rotations, and to prove that the kernel of the
linearized operator at a ground state consists exactly of the
infinitesimal symmetry directions.

A higher-order nonrelativistic expansion gives the relativistic corrections
which are invisible in the leading Hartree limit.  Of particular physical
interest is the expansion of \(e_c-mc^2\), which measures the successive
relativistic corrections to the ground state energy.  In the Dirac-Fock theory,
Esteban and S\'er\'e established the nonrelativistic limit of the
Dirac-Fock equations \cite{EstebanSere2001}.  More recently, Meng
\cite{Meng2025} obtained an \(O(c^{-2})\) correction to the ground state
energy and, for regular external potentials, identified the leading term in
terms of the mass-velocity, Darwin, and spin-orbit corrections.  For the 
pseudo-relativistic Hartree equation, all-order expansions were obtained in
\cite{ChenCotiZelatiWei2026}.  The Dirac-Coulomb problem is more involved:
the upper and lower spinor components and the Lagrange multiplier must be
expanded simultaneously, while the limiting linearization has a
six-dimensional symmetry kernel.  We overcome these difficulties by using
the augmented linearization and solve the
 linear equations recursively.

To describe the symmetry orbit of ground state to \eqref{eq:1.1}, for
$S\in\operatorname{Spin}(3)=SU(2)$, let $R_S\in SO(3)$ be determined by
\begin{equation}
 S(\sigma\cdot\xi)S^{-1}=:\sigma\cdot(R_S\xi),
 \qquad \forall ~\xi\in\R^3.
 \label{eq:1.9}
\end{equation}
Let
\begin{equation*}
 \cH:=(\R^3\rtimes\operatorname{Spin}(3))\times U(1),
\end{equation*}
here \(\rtimes\) means the action \(y'\mapsto R_Sy'\) of
\(\operatorname{Spin}(3)\) on the translation group.
For $\gamma=(y,S,e^{\ii\theta})\in\cH$ and
$u=(f,g)^T:\R^3\to\C^4$, define
\begin{equation*}
 (\cT_\gamma u)(x)
 :=e^{\ii\theta}
 \begin{pmatrix}S&0\\0&S\end{pmatrix}
 u\bigl(R_S^{-1}(x-y)\bigr).
\end{equation*}
The group law and the induced action on pairs of $\C^2$-valued
functions are given in \eqref{eq:2.4}.

Our first main result shows that the set of ground states forms a single orbit under
translations, phase transformations, and spinorial rotations.

\begin{theorem}
\label{thm:1.1}
There exists $c_0>0$, such that for $c>c_0$,
 the ground state of \eqref{eq:1.1}
is unique modulo translations, phase transformations, and spinorial
rotations.
More precisely,
\[
 \cG_c=\cH\cdot u_c,
\]
where $u_c=(f_c, g_c)^T$ is the unique representative with real radial function
$F_c,G_c$ satisfying
\begin{equation}\label{eq:1.10}
 \begin{gathered}
 f_c(x)=F_c(\abs x)\mathbf e_1,
 \qquad
 cg_c(x)=-\ii G_c(\abs x)(\sigma\cdot\widehat x)\mathbf e_1,\\
 F_c(r)>0\quad(r\ge0),
 \qquad G_c(0)=0,
 \qquad G_c(r)<0\quad(r>0).
 \end{gathered}
\end{equation}
Here $\widehat x=x/\abs x$ for $x\ne0$. 
\end{theorem}

\cref{thm:1.1} shows that $\cG_c$ is the orbit of $u_c$ under
the symmetry group $\mathcal H$ of the equation \eqref{eq:1.1}. We next prove
that the kernel of the constrained linearized operator at $u_c$
is spanned by the infinitesimal directions generated by these
symmetries.
For $s\in\R$, set
\begin{equation*}
 X_s:=H^s(\R^3,\C^2)\times H^s(\R^3,\C^2)\times\R.
\end{equation*}
Throughout the remainder of this section, let \(u_c=(f_c,g_c)^T\) denote the representative 
in \cref{thm:1.1}.  The real linearization of the
Dirac-Coulomb system under the $L^2$ constraint at
$(f_c,g_c,\lambda_c)$ is the operator
\begin{equation*}
 \cL_c:X_2\longrightarrow X_1
\end{equation*}
defined by
\begin{equation*}
\begin{split}
 \mathcal{L}_c(\varphi,\chi,\mu)
 = & \begin{pmatrix}
      c\mathcal{D}\chi+\lambda_c\varphi\\
      c\mathcal{D}\varphi+(\lambda_c-2mc^2)\chi\\
      2\ipR{f_c}{\varphi}+2\ipR{g_c}{\chi}
    \end{pmatrix} \quad - \bigl[|x|^{-1}\ast (|f_c|^2+|g_c|^2)\bigr]
    \begin{pmatrix}\varphi\\ \chi\\ 0\end{pmatrix}\\
 &
    + \bigl[\mu-2|x|^{-1}\ast \Re(f_c^\dagger\varphi+g_c^\dagger\chi)\bigr]
    \begin{pmatrix}f_c\\ g_c\\ 0\end{pmatrix}.
\end{split}
\end{equation*}

{
For \(u\in H^{1/2}(\R^3,\C^4)\), write
\begin{align*}
 \mathfrak q_c(u)
 :={\tfrac12}\mathrm d^2\cE_c(u_c)[u,u]
 -\omega_c\|u\|_{L^2}^2
\end{align*}

}

\begin{theorem}
\label{thm:1.2}
There exists $c_0>0$ such that, for every $c>c_0$, there holds
\begin{equation}
\begin{aligned}
 \Ker\cL_c
 &=\mathrm T_{u_c}(\cH\cdot u_c)\times\{0\}\\
 &=\Span_{\R}\left\{
 \begin{aligned}
  &(\partial_j f_c,\partial_j g_c,0),\quad 1\le j\le3,\\
  &(\ii f_c,\ii g_c,0),\\
  &\bigl(F_c(\abs x)\mathbf e_2,
    -\ii c^{-1}G_c(\abs x)(\sigma\cdot\widehat x)\mathbf e_2,0\bigr),\\
  &\bigl(\ii F_c(\abs x)\mathbf e_2,
    c^{-1}G_c(\abs x)(\sigma\cdot\widehat x)\mathbf e_2,0\bigr)
 \end{aligned}
 \right\}.
\end{aligned}
\label{eq:1.11}
\end{equation}
Moreover, there is \(\kappa>0\), independent of   \(c\), such that
\begin{equation}
 \mathfrak q_c(\xi)\ge\kappa\norm[c,+]{\xi}^2
 \quad\text{for }\quad
 \xi\in E_c^+\cap\{u_c\}^\perp
       \cap\bigl(\mathrm T_{u_c}(\cH\cdot u_c)\bigr)^\perp.
 \label{eq:1.12}
\end{equation}
where orthogonal complements are taken in real \(L^2\).
\end{theorem}

{
\begin{remark*}
\cref{thm:1.2} shows that the only zero modes of the constrained linearization
are those generated by the symmetries of the equation. Thus the ground state orbit is
nondegenerate.
For $w\in E_c^+\cap\mathcal{S}$, set
\[
 \mathcal J_c(w):=
 \sup_{\substack{u\in\cS\\P_c^+u\in\Span\{w\}}}\cE_c(u),
 \qquad
 w_c:=\frac{P_c^+u_c}{\norm[L^2]{P_c^+u_c}}.
\]
\eqref{eq:1.12} yields coercivity of the constrained Hessian of
\(\mathcal J_c\) transverse to \(\cH\cdot w_c\).
In particular, for some \(\delta>0\), uniformly for large \(c\),
\[
\inf_{\gamma\in\cH}\norm[c,+]{w-\cT_\gamma w_c}^2\lesssim  \mathcal J_c(w)-e_c
\]
provided $ \inf\limits_{\gamma\in\cH}\norm[c,+]{w-\cT_\gamma w_c}<\delta$, where
\[
\norm[c,+]{u}^2
 :=\norm[L^2]{u}^2+
 \ipR{(|\mathscr{D}_c|-mc^2)u}{u}.
\]
The proof is given at the end of Section~\ref{sec:3}.
This is local quantitative min-max stability, not positivity of
\(\cE_c|_{\cS}\), whose negative directions remain.
\end{remark*}
}

For $s\ge0$, let
\begin{align*}
 \cW^s_{\mathrm{up}}
 &:={\{F(\abs x)\mathbf e_1\in H^s(\R^3,\C^2):F\text{ is real}\}},\\
 \cW^s_{\mathrm{down}}
 &:={\{-\ii G(\abs x)(\sigma\cdot\widehat x)\mathbf e_1
       \in H^s(\R^3,\C^2):G\text{ is real}\}}.
\end{align*}
These are closed real subspaces of $H^s(\mathbb{R}^3, \mathbb{C}^2)$. The upper component is radial,
whereas the lower component has the displayed spinorial angular
structure. 

Set
\[
 \begin{gathered}
 h_c=cg_c,\qquad \lambda_c=mc^2-\omega_c,\qquad
 z_c=(f_c,h_c,\lambda_c,0),\\
z_\infty=(f_\infty,h_\infty,\lambda_\infty,0),
 \end{gathered}
\]
The next theorem gives the asymptotic expansions
of the representative ground state in Theorem~\ref{thm:1.1}.
\begin{theorem}\label{thm:1.3}
There exist unique 
\[
 z_n=(f_n,h_n,\lambda_n,0)
 \in\bigcap_{s\ge0}
 \bigl(\cW^s_{\mathrm{up}}\times\cW^s_{\mathrm{down}}
       \times\R\times\{0\}\bigr),
 \qquad n\ge1,
\]
and $R>0$, independent of $s\ge0$, such that, for all sufficiently
large $c$ with $c^{-2}<R$,
\begin{align}
 z_c
 &=z_\infty+\sum_{n=1}^{\infty}\frac{z_n}{c^{2n}}
 \quad\text{in }X_s\times\R^6,\qquad s\ge0,
 \label{eq:1.16}\\
 e_c-mc^2
 &=e_\infty+\sum_{n=1}^{\infty}\frac{\mathfrak e_n}{c^{2n}},
 \qquad
 \mathfrak e_n=-\frac{\lambda_n}{2n+3},
 \label{eq:1.17}
\end{align}

The coefficients $z_n$ are determined recursively by
\begin{align}
 \widehat{\cL}_\infty z_n
 ={}&
 -\sum_{k=2}^{\min\{3,n\}}\frac1{k!}
 \sum_{\substack{i_1+\cdots+i_k=n\\i_1,\ldots,i_k\ge1}}
 D_z^k\widehat{\cF}_\infty(z_\infty)
 [z_{i_1},\ldots,z_{i_k}]
 \notag\\
 &-\sum_{\ell=1}^{\min\{2,n\}}
 \sum_{k=0}^{\min\{3,n-\ell\}}\frac1{k!}
 \sum_{\substack{i_1+\cdots+i_k=n-\ell\\i_1,\ldots,i_k\ge1}}
 D_z^k\cC_\ell(z_\infty)
 [z_{i_1},\ldots,z_{i_k}],
 \label{eq:1.13}
\end{align}
where $\widehat{\cL}_\infty=D_z\widehat{\cF}_\infty(z_\infty)$
and $\cC_1,\cC_2$ are defined in
\eqref{eq:2.16} and \eqref{eq:4.2}, respectively.
For $k=0$, the last inner sum is $\cC_\ell(z_\infty)$
if $n=\ell$ and zero otherwise.
\end{theorem}

The paper is organized as follows. Section~\ref{sec:2} collects the
preliminary results.
Section~\ref{sec:3} proves Theorems~\ref{thm:1.1} and~\ref{thm:1.2}.
Section~\ref{sec:4} proves
Theorem~\ref{thm:1.3}.

\subsection*{Notations}
Throughout this paper, we make use of the following notations.
\begin{itemize}

\item $\|\cdot\|_{L^q}$ denotes the usual norm of
$L^q(\R^3,\C^N)$, where $N=1,\, 2$ or $4$;

\item $\|\cdot\|_{H^s}$ denotes the usual norm of
$H^s(\R^3,\C^N)$, where $N=1,\, 2$ or $4$;

\item $\mathrm T_u\mathcal M$ denotes the tangent space of a
smooth manifold $\mathcal M$ at $u\in\mathcal M$;

\item $D\cF(u)$ denotes the Fr\'echet derivative of a map
$\cF$ at $u$, while $\mathrm{d}\cE(u)$ denotes the Fr\'echet derivative
of a real-valued functional $\cE$ at $u$;

\item
$z^\dagger$ denotes the Hermitian transpose of
$z\in\C^N$;

\item for $f,g\in L^2(\R^3,\C^N)$,
\[
\ip{f}{g}
:=
\int_{\R^3} f(x)^\dagger g(x)\,\dd x,
\qquad
\ipR{f}{g}
:=
\Re\ip{f}{g};
\]

\item For a normed space $X$, and nonempty subsets
$A, B\subset X$, the distance between $A$ and $B$ is defined by
\[
\operatorname{dist}_X(A,B)
:=\inf_{v\in A, u\in B}\norm[X]{u-v};
\]

\item $\Re$ denotes the real part of a complex valued function;

\item $\mathcal B(X,Y)$ denotes the Banach space of all bounded linear operators from a Banach space $X$ to a Banach space $Y$, endowed with the operator norm
\[
\|T\|_{\mathcal B(X,Y)}
:=
\sup_{\|u\|_X=1}\|Tu\|_Y;
\]

\item $\Ker\cL$ denotes the kernel of a linear operator $\cL$;

\item $C>0$ denotes a constant whose value may change from line
to line;

\item $a\lesssim b$ means that $a\le Cb$ for some constant
$C>0$ independent of the relevant parameters. {We write
$a\asymp b$ when $a\lesssim b$ and $b\lesssim a$.}

\item For real derivatives, complex Sobolev spaces are regarded as real
Banach spaces. For a real Banach space $X$, set
\[
 \begin{gathered}
 X^\C=X\oplus\ii X,
 \qquad \overline{x+\ii y}=x-\ii y,\\
 \norm[X^\C]{x+\ii y}
 :=\sup_{\theta\in[0,2\pi]}
       \norm[X]{x\cos\theta-y\sin\theta}.
 \end{gathered}
\]
This conjugation is the one on the complexification, not pointwise
conjugation of the original spinor. A superscript $\C$ on a real
polynomial denotes its termwise complexification; in particular,
$\ipR{\cdot}{\cdot}^\C$ is complex bilinear.

\end{itemize}

\section{\texorpdfstring{{Preliminary results}}{Preliminary results}}
\label{sec:2}

We start by recalling briefly some basic facts concerning the 
 ground state to \eqref{eq:1.1}.

\begin{proposition}[\cite{Nolasco2021,ChenDingGuo2026}]\label{prop:2.1}
{
There exists \(c_0>0\) such that the following holds.

\begin{enumerate}[label=\textup{(\arabic*)}]
\item \emph{(Existence)}
For every \(c>c_0\), there is at least one
\(u_c=(f_c,g_c)^T\in\cG_c\), with a multiplier
\(\omega_c\in(0,mc^2)\).

\item \emph{(Regularity)}
For \(c>c_0\) 
\begin{equation*}
\cG_c\subset \bigcap_{s>0} H^s(\R^3,\C^4),
\end{equation*}
and, for every \(s>0\),
\begin{equation*}
 \sup_{c>c_0}\sup_{u_c\in \cG_c}
 \left(\norm[H^s]{f_c}+\norm[H^s]{cg_c}\right)<\infty.
\end{equation*}

\item \emph{(Nonrelativistic limit)} For every $s>0$,
\begin{equation*}
 mc^2-\omega_{c}\to\lambda_\infty,
 \qquad
 e_{c}-mc^2\to e_\infty,
\end{equation*}
\[
\sup_{u_c=(f_c,g_c)^T\in \cG_c} \operatorname{dist}_{H^s}\left(
f_c, \cG_\infty
\right)\to 0
\]
and
\[
\sup_{u_c=(f_c,g_c)^T\in \cG_c} \operatorname{dist}_{H^s}\left(
cg_c, \frac{1}{2m}\cD\cG_\infty
\right)\to 0
\]
as $c\to \infty$.

\end{enumerate}
}
\end{proposition}
\begin{remark*}
The limits of $\lambda_c$ and $e_c-mc^2$ hold for the whole family,
without passing to a subsequence.
\end{remark*}

\begin{lemma}
\label{lem:2.2}
For every \(s\ge\tfrac12\) and \(f,g,h\in H^s(\R^3)\),
\begin{equation}
 \norm[H^s]{\bigl(\abs{x}^{-1}\ast(fg)\bigr)h}
 \lesssim \norm[H^s]{f}
 \norm[H^s]{g}
 \norm[H^s]{h}.
 \label{eq:2.1}
\end{equation}
\end{lemma}

\begin{proof}
   {See \cite{ChoiHongSeok2018a}.}
\end{proof}

\begin{lemma}\label{lem:2.3}
For every \(s\in\R\) and \(f\in H^{s+1}(\R^3,\C^2)\),
\begin{equation}
 \norm[H^{s+1}]{f}
 \lesssim \norm[H^s]{\cD f}
 +\norm[H^s]{f}.
 \label{eq:2.2}
\end{equation}
\end{lemma}

\begin{proof}
Since \((\sigma\cdot\xi)^2=\abs\xi^2I_2\),
\[
 (1+\abs\xi^2)^{s+1}
 \lesssim (1+\abs\xi^2)^s\abs\xi^2+(1+\abs\xi^2)^s.
\]
Plancherel's theorem gives \eqref{eq:2.2}.
\end{proof}

{
\begin{lemma}\label{lem:2.4}
The complete set of minimizers in \eqref{eq:1.7} is
\begin{equation*}
 \cG_\infty
 =\{Q(\cdot-y)\xi:y\in\R^3,\ \xi\in\C^2,\ \abs\xi=1\}
\end{equation*}
which is homeomorphic to $\R^3\times\mathbb S^3$.
\end{lemma}
}

\begin{proof}
Let \(f\) be a minimizer.  The Kato inequality gives
\begin{equation*}
 \abs f\in H^1(\R^3,\R),
 \qquad
 \abs{\nabla\abs f}\le \abs{\nabla f}\quad\text{a.e.}
\end{equation*}
Consequently,
\begin{equation}
 e_\infty\le \cE_\infty(\abs f)\le \cE_\infty(f)=e_\infty.
 \label{eq:2.3}
\end{equation}
Then after a translation, \(\abs f=Q>0\).
Write
\begin{equation*}
 f=Q\zeta,
 \qquad \abs\zeta=1.
\end{equation*}
Since \(\Re(\zeta^\dagger\partial_j\zeta)=0\),
\begin{equation*}
 \abs{\nabla f}^2=\abs{\nabla Q}^2+Q^2\abs{\nabla\zeta}^2.
\end{equation*}
\eqref{eq:2.3} implies
\begin{equation*}
 \int_{\R^3}Q^2\abs{\nabla\zeta}^2\,\dd x=0,
\end{equation*}
so \(\zeta\equiv\xi\in\mathbb S^3\).  The converse follows from the
invariance of \(\cE_\infty\).
\end{proof}

An element of \(\cH\) is written as
\begin{equation*}
 \gamma=(y,S,e^{\ii\theta}),
 \qquad
 y\in\R^3,\quad
 S\in\operatorname{Spin}(3),\quad
 e^{i\theta}\in \mathbb S^1.
\end{equation*}
The group law and inverse are
\begin{equation}\label{eq:2.4}
 \begin{aligned}
 &(y,S,e^{\ii\theta})(y',S',e^{\ii\theta'})
 =\bigl(y+R_Sy',SS',e^{\ii(\theta+\theta')}\bigr),\\
 &(y,S,e^{\ii\theta})^{-1}
 =\bigl(-R_S^{-1}y,S^{-1},e^{-\ii\theta}\bigr).
 \end{aligned}
\end{equation}

For \(f:\R^3\to\C^2\),
\begin{equation*}
 (\cT_\gamma f)(x)
 :=e^{\ii\theta}S f\bigl(R_S^{-1}(x-y)\bigr),
\end{equation*}
and the action on pairs is diagonal:
\begin{equation*}
 \cT_\gamma(f,h):=(\cT_\gamma f,\cT_\gamma h).
\end{equation*}
Consequently,
\begin{equation*}
 \begin{aligned}
 \cH\cdot(f_\infty,h_\infty)
 =\Bigl\{&\left(
 Q(\abs{x-y})\xi,
 -\frac{\ii}{2m}Q'(\abs{x-y})
 (\sigma\cdot\widehat{x-y})\xi
 \right):\\
 &y\in\R^3,\ \xi\in\C^2,\ \abs\xi=1\Bigr\}.
 \end{aligned}
\end{equation*}
Its stabilizer is
\begin{equation*}
 \operatorname{Stab}_{\cH}(f_\infty,h_\infty)
 =\left\{
 \bigl(0,e^{-\ii\vartheta\sigma_3/2},e^{\ii\vartheta/2}\bigr):
 \vartheta\in\R
 \right\},
\end{equation*}
and hence
\begin{equation*}
 \dim\bigl(\cH\cdot(f_\infty,h_\infty)\bigr)=6.
\end{equation*}
The real linearization of the limiting
Hartree equation at \(f_\infty\) is
\begin{equation*}
 \begin{aligned}
 \cL:H^2(\R^3,\C^2)&\longrightarrow L^2(\R^3,\C^2),\\
 \cL\varphi
 &=\left(-\frac1{2m}\Delta+\lambda_\infty
 -\abs{x}^{-1}\ast Q^2\right)\varphi
 -2\left(\abs{x}^{-1}\ast \Re(f_\infty^\dagger\varphi)\right)f_\infty.
 \end{aligned}
\end{equation*}
The associated real operators are
\begin{align*}
 \cL_-P
 &:=-\frac1{2m}\Delta P+\lambda_\infty P
   -(\abs{x}^{-1}\ast Q^2)P,
 \\
 \cL_+P
 &:=\cL_-P-2Q\bigl(\abs{x}^{-1}\ast(QP)\bigr).
\end{align*}

Set
\begin{equation}
 \begin{gathered}
 v_j:=\partial_jQ\mathbf e_1\quad(1\le j\le3),
 \qquad
 v_4:=\ii Q\mathbf e_1,
 \qquad
 v_5:=Q\mathbf e_2,
 \qquad
 v_6:=\ii Q\mathbf e_2,\\
 \cK:=\Span_{\R}\{v_1,\ldots,v_6\},
 \qquad
 \cK^\perp
 :=\{f\in L^2(\R^3,\C^2):\ipR{f}{v_j}=0,
 \ 1\le j\le6\}.
 \end{gathered}
 \label{eq:2.5}
\end{equation}

{
\begin{lemma}\label{lem:2.5}
\(\cL\) is self-adjoint and Fredholm of index zero as a real operator,
\begin{equation}
 \Ker_{\R}\cL=\cK=\mathrm T_{f_\infty}\cG_\infty.
 \label{eq:2.6}
\end{equation}
For every \(s\ge2\),
\begin{equation}
 \cL:
 H^s(\R^3,\C^2)\cap\cK^\perp
 \longrightarrow
 H^{s-2}(\R^3,\C^2)\cap\cK^\perp
 \label{eq:2.7}
\end{equation}
is an isomorphism. 
\end{lemma}

}

\begin{proof}
Write
\begin{equation*}
 \varphi=(\Phi_1+\ii\Phi_2,\Phi_3+\ii\Phi_4)^T,
 \qquad \Phi_1,\ldots,\Phi_4:\R^3\to\R.
\end{equation*}
Then
\begin{equation*}
 \cL\varphi
 =\bigl(\cL_+\Phi_1+\ii\cL_-\Phi_2,
         \cL_-\Phi_3+\ii\cL_-\Phi_4\bigr)^T.
\end{equation*}
Nondegeneracy of $\cL_+$ gives
\begin{equation*}
 \Ker\cL_+=\Span_{\R}\{\partial_1Q,\partial_2Q,\partial_3Q\},
 \qquad
 \Ker\cL_-=\Span_{\R}\{Q\};
\end{equation*}
see \cite{Lenzmann2009}. Hence
\begin{equation*}
 \Ker_{\R}\cL
 =\Span_{\R}\{\partial_1Q\mathbf e_1,
 \partial_2Q\mathbf e_1,\partial_3Q\mathbf e_1,
 \ii Q\mathbf e_1,Q\mathbf e_2,\ii Q\mathbf e_2\}
 =\cK.
\end{equation*}
Since
\begin{equation*}
 \cG_\infty
 =\{Q(\cdot-y)\xi:y\in\R^3,\ \xi\in\C^2,\ \abs{\xi}=1\}
\end{equation*}
and
\begin{equation*}
 \mathrm T_{\mathbf e_1}\mathbb S^3
 =\Span_{\R}\{\ii\mathbf e_1,\mathbf e_2,\ii\mathbf e_2\},
\end{equation*}
differentiation at \((y,\xi)=(0,\mathbf e_1)\) gives
\begin{equation*}
 \mathrm T_{f_\infty}\cG_\infty=\cK.
\end{equation*}
This proves \eqref{eq:2.6}.

The operator \(\cL_-\) and \(\cL_+\) are self-adjoint on
\(L^2(\R^3,\R)\) with domain \(H^2(\R^3,\R)\).  As bounded operators
from \(H^2(\R^3,\R)\) to \(L^2(\R^3,\R)\), they are compact
perturbations of the isomorphism
\begin{equation*}
 -\frac1{2m}\Delta+\lambda_\infty:
 H^2(\R^3,\R)\longrightarrow L^2(\R^3,\R).
\end{equation*}
Thus
\begin{equation}
 \operatorname{ind}\cL_-=\operatorname{ind}\cL_+=0,
 \label{eq:2.8}
\end{equation}
 its range is closed and
\begin{equation}
 \Ran\cL=(\Ker_{\R}\cL)^\perp=\cK^\perp.
 \label{eq:2.9}
\end{equation}
For $r\in H^{s-2}\cap\cK^\perp$, $s\ge2$, let
$\varphi\in H^2\cap\cK^\perp$ solve $\cL\varphi=r$. Then
\[
 \left(-\frac1{2m}\Delta+\lambda_\infty\right)\varphi
 =r+(\abs{x}^{-1}\ast Q^2)\varphi
   +2\bigl(\abs{x}^{-1}\ast\Re(f_\infty^\dagger\varphi)\bigr)f_\infty.
\]
The fixed coefficients are smooth. Lemma~\ref{lem:2.2} and
elliptic bootstrapping give
\[
 \varphi\in H^s\cap\cK^\perp,
 \qquad
 \norm[H^s]{\varphi}
 \le C_s\bigl(\norm[H^{s-2}]{r}+\norm[L^2]{\varphi}\bigr)
 \le C_s\norm[H^{s-2}]{r}.
\]
Together with \eqref{eq:2.6}, this proves \eqref{eq:2.7}.

\end{proof}
{
\begin{lemma}\label{lem:2.6}
For every \(s\ge2\), define
\begin{equation*}
 \cL_{\mathrm{mf}}(\varphi,\mu)
 :=\left(\cL\varphi+\mu f_\infty,\ 2\ipR{f_\infty}{\varphi}\right).
\end{equation*}
Then
\begin{equation*}
 \cL_{\mathrm{mf}}:
 \bigl(H^s(\R^3,\C^2)\cap\cK^\perp\bigr)\times\R
 \longrightarrow
 \bigl(H^{s-2}(\R^3,\C^2)\cap\cK^\perp\bigr)\times\R
\end{equation*}
is an isomorphism. 
\end{lemma}
}

\begin{proof}
Set
\begin{equation*}
 \cR
 :=\left(
 \restr{\cL}{H^2(\R^3,\C^2)\cap\cK^\perp}
 \right)^{-1}.
\end{equation*}
By \cref{lem:2.5}, the inverse is bounded from
\(H^{s-2}(\R^3,\C^2)\cap\cK^\perp\) to
\(H^s(\R^3,\C^2)\cap\cK^\perp\) for every \(s\ge2\).
Since \(f_\infty\in\cK^\perp\), it remains to compute
\(\ipR{f_\infty}{\cR f_\infty}\).  For \(\lambda>0\), set
\begin{equation*}
 Q_\lambda(x)
 :=\frac{\lambda}{\lambda_\infty}
 Q\!\left(\sqrt{\frac{\lambda}{\lambda_\infty}}\,x\right).
\end{equation*}
Then
\[
 -\frac1{2m}\Delta Q_\lambda+\lambda Q_\lambda
 -(\abs{x}^{-1}\ast Q_\lambda^2)Q_\lambda=0,
 \qquad
 \norm[L^2]{Q_\lambda}^2
 =\left(\frac{\lambda}{\lambda_\infty}\right)^{1/2}.
\]
Differentiation at \(\lambda=\lambda_\infty\) gives
\begin{equation*}
 \cL_+\left(\partial_\lambda Q_\lambda
 \big|_{\lambda=\lambda_\infty}\right)=-Q.
\end{equation*}
Hence
\begin{equation}
 \begin{aligned}
 d:=\ipR{f_\infty}{\cR f_\infty}
 &=\ip{Q}{\cL_+^{-1}Q}\\
 &=-\ip{Q}{\partial_\lambda Q_\lambda}
       \big|_{\lambda=\lambda_\infty}\\
 &=-\frac12\frac{\dd}{\dd\lambda}
   \norm[L^2]{Q_\lambda}^2
       \bigg|_{\lambda=\lambda_\infty}
 =-\frac1{4\lambda_\infty}\ne0.
 \end{aligned}
 \label{eq:2.10}
\end{equation}
For
\((r,\ell)\in(H^{s-2}(\R^3,\C^2)\cap\cK^\perp)\times\R\), set
\begin{equation*}
 \mu:=\frac{2\ipR{f_\infty}{\cR r}-\ell}{2d},
 \qquad
 \varphi:=\cR(r-\mu f_\infty).
\end{equation*}
Then
\[
 \cL_{\mathrm{mf}}(\varphi,\mu)=(r,\ell),
\]
 The same formula proves
injectivity.
\end{proof}

By \eqref{eq:2.5}, we have 
\begin{equation*}
 \ipR{f_\infty}{v_j}=0,
 \qquad 1\le j\le6.
\end{equation*}
Define the fully augmented operator
\begin{equation*}
 \cL_{\mathrm{aug}}(\varphi,\mu,\bm b)
 :=\left(
 \cL\varphi+\mu f_\infty+\sum_{j=1}^6b_jv_j,
 \ 2\ipR{f_\infty}{\varphi},
 \ (\ipR{\varphi}{v_j})_{j=1}^6
 \right).
\end{equation*}

{
\begin{lemma}\label{lem:2.7}
For every \(s\ge2\),
\begin{equation*}
 \cL_{\mathrm{aug}}:
 H^s(\R^3,\C^2)\times\R\times\R^6
 \longrightarrow
 H^{s-2}(\R^3,\C^2)\times\R\times\R^6
\end{equation*}
is a bounded isomorphism.
\end{lemma}
}

\begin{proof}
Let
\begin{equation}
 \Gamma_{ij}:=\ipR{v_j}{v_i},
 \qquad 1\le i,j\le6.
 \label{eq:2.11}
\end{equation}
The Gram matrix \(\Gamma\) is positive definite.  For
\((r,\ell,\bm d)\in H^{s-2}(\R^3,\C^2)\times\R\times\R^6\), put
\begin{equation*}
 r_i:=\ipR{r}{v_i},
 \qquad
 \bm b:=\Gamma^{-1}(r_1,\ldots,r_6)^T,
 \qquad
 r_\perp:=r-\sum_{j=1}^6b_jv_j.
\end{equation*}
Then \(r_\perp\in\cK^\perp\).  By
\cref{lem:2.6}, there is a unique
\((\varphi_\perp,\mu)\in
\bigl(H^s(\R^3,\C^2)\cap\cK^\perp\bigr)\times\R\) satisfying
\begin{equation*}
 \cL_{\mathrm{mf}}(\varphi_\perp,\mu)=(r_\perp,\ell).
\end{equation*}
Set
\begin{equation*}
 \bm a:=\Gamma^{-1}\bm d,
 \qquad
 \varphi:=\varphi_\perp+\sum_{j=1}^6a_jv_j.
\end{equation*}
Since \(\cL v_j=0\) and \(f_\infty\perp\cK\),
\(\cL_{\mathrm{aug}}(\varphi,\mu,\bm b)=(r,\ell,\bm d)\).
The construction is unique, and
\begin{equation*}
 \norm[H^s]{\varphi}+\abs\mu+\abs{\bm b}
 \lesssim \norm[H^{s-2}]{r}+\abs\ell+\abs{\bm d}.
\end{equation*}
\end{proof}

For finite \(c>0\), define
\begin{equation}
 \cF_c(f,h,\lambda)
 :=\begin{pmatrix}
  \cD h+
  \left\{\lambda-\abs{x}^{-1}\ast
  \left(\abs f^2+c^{-2}\abs h^2\right)\right\}f\\[0.4em]
  \cD f-2mh+c^{-2}
  \left\{\lambda-\abs{x}^{-1}\ast
  \left(\abs f^2+c^{-2}\abs h^2\right)\right\}h\\[0.4em]
  \norm[L^2]{f}^2+c^{-2}\norm[L^2]{h}^2-1
 \end{pmatrix},
 \label{eq:2.12}
\end{equation}
and
\begin{equation*}
 \cF_\infty(f,h,\lambda)
 :=\begin{pmatrix}
  \cD h+
  \left(\lambda-\abs{x}^{-1}\ast \abs f^2\right)f\\[0.4em]
  \cD f-2mh\\[0.4em]
  \norm[L^2]{f}^2-1
 \end{pmatrix}.
\end{equation*}
For
\begin{equation*}
 \ell_j(f):=\ipR{f-f_\infty}{v_j},
 \qquad 1\le j\le6,
\end{equation*}
and \(\bm a=(a_1,\ldots,a_6)\in\R^6\), put
\begin{equation}
 \widehat{\cF}_c(f,h,\lambda,\bm a)
 :=\left(
 \cF_c(f,h,\lambda)
 +\begin{pmatrix}
   \displaystyle\sum_{j=1}^6a_jv_j\\[0.2em]0\\[0.2em]0
  \end{pmatrix},
 \ (\ell_j(f))_{j=1}^6
 \right).
 \label{eq:2.13}
\end{equation}
The map \(\widehat{\cF}_\infty\) is obtained from \eqref{eq:2.13} by
setting \(c^{-2}=0\).

{
\begin{lemma}\label{lem:2.8}
For every \(s\ge\tfrac12\) and every \(c>c_0\),
\begin{equation*}
 \cF_c:X_s\longrightarrow X_{s-1},
 \qquad
 \widehat{\cF}_c:X_s\times\R^6
 \longrightarrow X_{s-1}\times\R^6
\end{equation*}
are \(C^\infty\) real maps.  For every \(R>0\),
\begin{align}
 &\sup_{\norm[X_s\times\R^6]{z-z_\infty}\le R}
 \norm[X_{s-1}\times\R^6]{
 \widehat{\cF}_c(z)-\widehat{\cF}_\infty(z)}
 \lesssim c^{-2},
 \label{eq:2.14}\\
 &\sup_{\norm[X_s\times\R^6]{z-z_\infty}\le R}
 \norm[\cB(X_s\times\R^6,X_{s-1}\times\R^6)]{
 D_z\widehat{\cF}_c(z)-D_z\widehat{\cF}_\infty(z)}
 \lesssim c^{-2}.
 \label{eq:2.15}
\end{align}
\end{lemma}
}

\begin{proof}
The operator \(\cD\) is bounded from \(H^s\) to \(H^{s-1}\), and
\eqref{eq:2.1} gives
\[
 \norm[H^s]{(\abs{x}^{-1}\ast \Re(f^\dagger g))h}
 \lesssim \norm[H^s]{f}\norm[H^s]{g}\norm[H^s]{h}.
\]
Thus every nonlinear term in \eqref{eq:2.12} and \eqref{eq:2.13} is a
continuous map from \(H^s\times H^s\) to \(H^s\), hence also
to \(H^{s-1}\).  Moreover,
\begin{equation*}
 \cF_c(f,h,\lambda)-\cF_\infty(f,h,\lambda)
 =c^{-2}\begin{pmatrix}
  -(\abs{x}^{-1}\ast \abs h^2)f\\[0.35em]
  \left\{\lambda-\abs{x}^{-1}\ast \abs f^2
  -c^{-2}\abs{x}^{-1}\ast \abs h^2\right\}h\\[0.35em]
  \norm[L^2]{h}^2
 \end{pmatrix}.
\end{equation*}
The same estimates applied to the derivatives prove
\eqref{eq:2.14}-\eqref{eq:2.15}.
\end{proof}

Define
\begin{equation}
 \begin{aligned}
 &\widehat{\cL}_\infty
    :=D_z\widehat{\cF}_\infty(z_\infty),\\
 &\widehat{\cL}_\infty(\varphi,\psi,\mu,\bm b)
 =\begin{pmatrix}
 \begin{aligned}
 \cD\psi
 &+\left(\lambda_\infty-\abs{x}^{-1}\ast Q^2\right)\varphi\\[-0.2em]
 &-2\bigl(\abs{x}^{-1}\ast\Re(f_\infty^\dagger\varphi)\bigr)f_\infty
       +\mu f_\infty+\displaystyle\sum_{j=1}^6b_jv_j
 \end{aligned}\\[0.6em]
 \cD\varphi-2m\psi\\[0.2em]
 2\ipR{f_\infty}{\varphi}\\[0.2em]
 (\ipR{\varphi}{v_j})_{j=1}^6
 \end{pmatrix}.
 \end{aligned}
 \label{eq:2.16}
\end{equation}

\begin{lemma}\label{lem:2.9}
For every $s\ge0$,
\begin{equation}
 \widehat{\cL}_\infty:X_s\times\R^6
       \longrightarrow X_{s-1}\times\R^6
 \label{eq:2.17}
\end{equation}
is an isomorphism. 
For every $s\ge1$, its restriction is an isomorphism
\begin{equation}
 \begin{aligned}
 \widehat{\cL}_\infty:
 &\cW^s_{\mathrm{up}}\times\cW^s_{\mathrm{down}}
                         \times\R\times\{0\}\\
 &\hspace{1em}\longrightarrow
 \cW^{s-1}_{\mathrm{up}}\times\cW^{s-1}_{\mathrm{down}}
                         \times\R\times\{0\}.
 \end{aligned}
 \label{eq:2.19}
\end{equation}
The nonlinear maps $\widehat{\cF}_c$, including $c=\infty$,
map the first space in \eqref{eq:2.19} into the second.
\end{lemma}

\begin{proof}
Let $s\ge2$ and $(r_1,r_2,\ell,\bm d)\in X_{s-1}\times\R^6$.
By Lemma~\ref{lem:2.7}, there is a unique
$(\varphi,\mu,\bm b)\in H^s(\R^3,\C^2)\times\R\times\R^6$ such that
\[
 \begin{gathered}
 \cL_{\mathrm{aug}}(\varphi,\mu,\bm b)
   =\left(r_1+\frac1{2m}\cD r_2,\ell,\bm d\right),\\
 \norm[H^s]{\varphi}+|\mu|+|\bm b|
 \le C_s\bigl(\norm[H^{s-1}]{r_1}
       +\norm[H^{s-1}]{r_2}+|\ell|+|\bm d|\bigr).
 \end{gathered}
\]
Set $\psi=(2m)^{-1}(\cD\varphi-r_2)$. Then
\[
 \begin{aligned}
 \cD\psi={}&r_1-
   (\lambda_\infty-\abs{x}^{-1}\ast Q^2)\varphi
   +2\bigl(\abs{x}^{-1}\ast\Re(f_\infty^\dagger\varphi)\bigr)f_\infty\\
 &-\mu f_\infty-\sum_{j=1}^6b_jv_j,\\
 \norm[H^s]{\psi}
 &\le C_s\bigl(\norm[H^{s-1}]{\psi}
                    +\norm[H^{s-1}]{\cD\psi}\bigr)\\
 &\le C_s\bigl(\norm[H^{s-1}]{r_1}
       +\norm[H^{s-1}]{r_2}+|\ell|+|\bm d|\bigr).
 \end{aligned}
\]
Thus \eqref{eq:2.17} hold for $s\ge2$.

For the real distributional adjoint,
\[
 \widehat{\cL}_\infty^*
 =\diag(I_2,I_2,\tfrac12,I_6)\,
       \widehat{\cL}_\infty\,
       \diag(I_2,I_2,2,I_6),
 \qquad
 (X_t\times\R^6)^*=X_{-t}\times\R^6.
\]
Duality at $s=2$ therefore gives the isomorphism
$\widehat{\cL}_\infty:X_{-1}\times\R^6\to X_{-2}\times\R^6$.
The inverse agrees with the preceding inverses on smooth data.
Interpolation, using
\[
 [X_{-1},X_2]_\theta=X_{3\theta-1},
 \qquad
 [X_{-2},X_1]_\theta=X_{3\theta-2},
\]
proves \eqref{eq:2.17} also for $0\le s<2$.

For real radial $F,G$, direct computation gives
\begin{equation}
 \begin{aligned}
 \cD(F(r)\mathbf e_1)
 &=-\ii F'(r)(\sigma\cdot\widehat x)\mathbf e_1,\\
 \cD\bigl(-\ii G(r)(\sigma\cdot\widehat x)\mathbf e_1\bigr)
 &=-\left(G'(r)+\frac{2G(r)}r\right)\mathbf e_1,\\
 \cW^0_{\mathrm{up}}&\subset\cK^\perp.
 \end{aligned}
 \label{eq:2.20}
\end{equation}
Moreover, $|F(r)\mathbf e_1|^2=F(r)^2$ and
$|-\ii G(r)(\sigma\cdot\widehat x)\mathbf e_1|^2=G(r)^2$.
Thus radial Coulomb convolution and \eqref{eq:2.20} give the
asserted invariance of $\widehat{\cF}_c$ and
$\widehat{\cL}_\infty$.

It remains to prove invariance of the inverse. First let $s\ge2$ and
\[
 (r_1,r_2,\ell,0)\in
 \cW^{s-1}_{\mathrm{up}}\times\cW^{s-1}_{\mathrm{down}}
                         \times\R\times\{0\}.
\]
For $(\varphi,\psi,\mu,\bm b)
 =\widehat{\cL}_\infty^{-1}(r_1,r_2,\ell,0)$,
\[
 \begin{gathered}
 \cL\varphi+\mu f_\infty+\sum_{j=1}^6b_jv_j
     =r_1+\frac1{2m}\cD r_2\in\cK^\perp,\\
 2\ipR{f_\infty}{\varphi}=\ell,
 \qquad \ipR{\varphi}{v_j}=0,\\
 \Gamma\bm b=0,
 \qquad \bm b=0.
 \end{gathered}
\]
Write $\varphi=(\Phi_1+\ii\Phi_2,\Phi_3+\ii\Phi_4)^T$.
The last three real components satisfy
\[
 \cL_-\Phi_j=0,
 \qquad \ip{Q}{\Phi_j}=0\quad(j=2,3,4),
 \qquad \Phi_2=\Phi_3=\Phi_4=0.
\]
For every $O\in SO(3)$, radiality of the first-component data gives
\[
 \Phi_1\circ O^{-1}-\Phi_1
 \in\Ker\cL_+\cap
 \Span_\R\{\partial_1Q,\partial_2Q,\partial_3Q\}^{\perp}
 =\{0\}.
\]
Consequently,
\[
 \varphi\in\cW^s_{\mathrm{up}},
 \qquad
 \psi=\frac1{2m}(\cD\varphi-r_2)
       \in\cW^{s-1}_{\mathrm{down}}\cap H^s
       =\cW^s_{\mathrm{down}}.
\]
For $1\le s<2$, replace $r_i$ by $e^{t\Delta}r_i$.
The heat semigroup preserves both angular structures and
\[
 e^{t\Delta}r_i\longrightarrow r_i\quad\text{in }H^{s-1}
 \qquad(t\to0^+).
\]
Closedness of the two subspaces give
\eqref{eq:2.19} at these indices as well.
\end{proof}

{
\begin{lemma}\label{lem:2.10}
There exist \(c_1>0\) and \(r>0\) such that, for every \(c>c_1\), there
is a unique
\begin{equation*}
 \widetilde z_c:=(f_c,h_c,\lambda_c,\bm a_c)
 \in B_r(z_\infty)\subset X_2\times\R^6
\end{equation*}
satisfying
\begin{equation}
 \widehat{\cF}_c(\widetilde z_c)=0.
 \label{eq:2.21}
\end{equation}
Moreover,
\begin{equation}
 \norm[H^2]{f_c-f_\infty}
 +\norm[H^2]{h_c-h_\infty}
 +\abs{\lambda_c-\lambda_\infty}+\abs{\bm a_c}
 \lesssim c^{-2}.
 \label{eq:2.22}
\end{equation}
\end{lemma}
}

\begin{proof}
By Lemma~\ref{lem:2.9},
$\widehat{\cL}_\infty:X_2\times\R^6\to X_1\times\R^6$
is an isomorphism, and $\widehat{\cF}_\infty(z_\infty)=0$.
A direct computation gives
\begin{equation}
 \widehat{\cF}_c(z_\infty)
 =\begin{pmatrix}
  -c^{-2}(\abs{x}^{-1}\ast \abs{h_\infty}^2)f_\infty\\[0.2em]
  c^{-2}(\lambda_\infty-\abs{x}^{-1}\ast Q^2)h_\infty
  -c^{-4}(\abs{x}^{-1}\ast \abs{h_\infty}^2)h_\infty\\[0.2em]
  c^{-2}\norm[L^2]{h_\infty}^2\\[0.2em]
  0
 \end{pmatrix}.
 \label{eq:2.23}
\end{equation}
Consequently,
\begin{equation*}
 \norm[X_1\times\R^6]{\widehat{\cF}_c(z_\infty)}\lesssim c^{-2}.
\end{equation*}
By \eqref{eq:2.15}, choose \(r>0\) and then \(c_1\) so that
\begin{equation}
 \sup_{z\in B_r(z_\infty)}
 \norm[\cB(X_2\times\R^6)]{
 \Id-\widehat{\cL}_\infty^{-1}D_z\widehat{\cF}_c(z)}
 \le\frac12,
 \qquad c>c_1.
 \label{eq:2.24}
\end{equation}
Set
\begin{equation*}
 \Lambda_c(z):=z-\widehat{\cL}_\infty^{-1}\widehat{\cF}_c(z).
\end{equation*}
For sufficiently large \(M\), \eqref{eq:2.23} and
\eqref{eq:2.24} give
{
\begin{equation*}
 z\in \overline B_{Mc^{-2}}(z_\infty)
 \quad\Longrightarrow\quad
 \norm[X_2\times\R^6]{\Lambda_c(z)-z_\infty}
 \le\frac12\norm[X_2\times\R^6]{z-z_\infty}+Cc^{-2}
 \le Mc^{-2}.
\end{equation*}
Thus \(\Lambda_c\) maps the complete metric space
\(\overline B_{Mc^{-2}}(z_\infty)\) into itself and is a contraction there.
It has a unique fixed point in this closed ball. 
\eqref{eq:2.24} gives uniqueness in \(B_r(z_\infty)\),
and \eqref{eq:2.22} follows.
}
\end{proof}

On $\cW^2_{\mathrm{up}}\times\cW^2_{\mathrm{down}}\times\R$,
\eqref{eq:1.2} becomes
\begin{equation}
 \left\{
 \begin{aligned}
 &\left\{\lambda_c-
 [\abs{x}^{-1}\ast(F_c^2+c^{-2}G_c^2)](r)\right\}F_c
       -G_c'-\frac{2G_c}{r}=0,\\
 &F_c'-\left\{2m-c^{-2}\lambda_c+
 c^{-2}[\abs{x}^{-1}\ast(F_c^2+c^{-2}G_c^2)](r)\right\}G_c=0,\\
 &4\pi\int_0^\infty(F_c^2+c^{-2}G_c^2)r^2\,\dd r=1.
 \end{aligned}
 \right.
 \label{eq:2.25}
\end{equation}
The limiting profiles are
\begin{equation}
 (F_\infty,G_\infty,\lambda_\infty)
 =\left(Q,\frac{Q'}{2m},\lambda_\infty\right).
 \label{eq:2.26}
\end{equation}

\begin{proposition}\label{prop:2.11}
For all sufficiently large $c$, \eqref{eq:2.25} has a unique
solution near \eqref{eq:2.26} in
$\cW^2_{\mathrm{up}}\times\cW^2_{\mathrm{down}}\times\R$.
It is the branch of Lemma~\ref{lem:2.10}, and
\begin{equation}
 \begin{gathered}
 f_c(x)=F_c(\abs x)\mathbf e_1,
 \qquad h_c(x)=-\ii G_c(\abs x)(\sigma\cdot\widehat x)\mathbf e_1,\\
 \bm a_c=0,
 \qquad \ipR{f_c}{f_\infty}>0.
 \end{gathered}
 \label{eq:2.27}
\end{equation}
\end{proposition}
\begin{proof}
By Lemma~\ref{lem:2.9}, the argument of Lemma~\ref{lem:2.10}
applies to
\[
 \Lambda_c(z):=
 z-\widehat{\cL}_\infty^{-1}\widehat{\cF}_c(z)
\]
on the closed set
\[
 \overline B_{Mc^{-2}}(z_\infty)
 \cap
 \left(
 \cW^2_{\mathrm{up}}\times\cW^2_{\mathrm{down}}
 \times\R\times\{0\}
 \right)
 \subset X_2\times\R^6.
\]
For fixed sufficiently large $M$ and all sufficiently large $c$,
$\Lambda_c$ maps this set into itself and is a contraction.
Its unique fixed point coincides with $\widetilde z_c$ from
Lemma~\ref{lem:2.10}. Hence
\[
 (f_c,h_c)\in
 \cW^2_{\mathrm{up}}\times\cW^2_{\mathrm{down}},
 \qquad
 \bm a_c=0,
 \qquad
 \cF_c(f_c,h_c,\lambda_c)=0.
\]
Thus the profiles satisfy \eqref{eq:2.25}, and local uniqueness
follows from Lemma~\ref{lem:2.10}. Finally,
\[
 \ipR{f_c}{f_\infty}
 \ge 1-\norm[L^2]{f_c-f_\infty}
 \ge 1-Cc^{-2}>0.
 \qedhere
\]
\end{proof}

By the standard bootstrapping argument, 
we obtain that the solution constructed in \cref{prop:2.11} converges in each
$H^s$.
{
\begin{lemma}
\label{lem:2.12}
For every \(s\ge\tfrac12\), the local radial branch satisfies
\begin{equation*}
 f_c,h_c\in H^s(\R^3,\C^2),
 \qquad
 (f_c,h_c,\lambda_c,0)\longrightarrow z_\infty
 \quad\text{in }X_s\times\R^6.
\end{equation*}
Moreover, for every \(s\ge2\),
\begin{equation*}
 \sup_{c\ge c_s}
 \left(\norm[H^s]{f_c}+\norm[H^s]{h_c}\right)<\infty.
\end{equation*}
\end{lemma}

}

{
\begin{lemma}
\label{lem:2.13}
Let \(s\ge\tfrac12\), \(z=(f,h,\lambda,\bm a)\in X_s\times\R^6\), and
\((\varphi,\psi,\mu,\bm b)\in X_s\times\R^6\).  Then
\begin{equation}
\begin{aligned}
&D_z\widehat{\cF}_c(z)(\varphi,\psi,\mu,\bm b)\\[-0.2em]
&\quad=
\left(
\begin{array}{c}
\begin{aligned}
 \cD\psi
 &+\Bigl[\lambda-\abs{x}^{-1}\ast
       \bigl(\abs f^2+c^{-2}\abs h^2\bigr)\Bigr]\varphi\\[-0.1em]
 &+\Bigl[\mu-2\abs{x}^{-1}\ast \Re
       \bigl(f^\dagger\varphi+c^{-2}h^\dagger\psi\bigr)\Bigr]f
 +\displaystyle\sum_{j=1}^6b_jv_j
\end{aligned}
\\[0.8em]
\begin{aligned}
 \cD\varphi-2m\psi
 &+c^{-2}\Bigl[\lambda-\abs{x}^{-1}\ast
       \bigl(\abs f^2+c^{-2}\abs h^2\bigr)\Bigr]\psi\\[-0.1em]
 &+c^{-2}\Bigl[\mu-2\abs{x}^{-1}\ast \Re
       \bigl(f^\dagger\varphi+c^{-2}h^\dagger\psi\bigr)\Bigr]h
\end{aligned}
\\[0.7em]
 2\ipR{f}{\varphi}+2c^{-2}\ipR{h}{\psi}
\\[0.35em]
 (\ipR{\varphi}{v_j})_{j=1}^6
\end{array}
\right).
\end{aligned}
\label{eq:2.28}
\end{equation}
It is a bounded operator from
\(X_s\times\R^6\) to \(X_{s-1}\times\R^6\).
For every \(s\ge2\), there are \(c_s,\delta_s,C_s>0\) such that
\begin{equation}
 \norm[X_s\times\R^6]{z-z_\infty}<\delta_s,
 \quad c\ge c_s
 \quad\Longrightarrow\quad
 \norm[\cB(X_{s-1}\times\R^6,X_s\times\R^6)]{
 D_z\widehat{\cF}_c(z)^{-1}}\le C_s.
 \label{eq:2.29}
\end{equation}
\end{lemma}
}

\begin{proof}
Differentiation gives \eqref{eq:2.28}; boundedness follows from
Lemma~\ref{lem:2.2}. By \eqref{eq:2.15},
for $s\ge2$ and $\norm[X_s\times\R^6]{z-z_\infty}\le1$,
\begin{equation}
 \norm[\cB(X_s\times\R^6,X_{s-1}\times\R^6)]{
 D_z\widehat{\cF}_c(z)-\widehat{\cL}_\infty}
 \le C_s\bigl(c^{-2}+\norm[X_s\times\R^6]{z-z_\infty}\bigr).
 \label{eq:2.30}
\end{equation}
Lemma~\ref{lem:2.9} allows $c_s,\delta_s$ to be chosen so that
\[
 \norm[\cB(X_s\times\R^6)]{
 \widehat{\cL}_\infty^{-1}
       (D_z\widehat{\cF}_c(z)-\widehat{\cL}_\infty)}\le\frac12.
\]
Thus
\[
 \begin{aligned}
 (D_z\widehat{\cF}_c(z))^{-1}
 &=\sum_{k=0}^\infty
 \left[-\widehat{\cL}_\infty^{-1}
       (D_z\widehat{\cF}_c(z)-\widehat{\cL}_\infty)\right]^k
       \widehat{\cL}_\infty^{-1},\\
 \norm[\cB(X_{s-1}\times\R^6,X_s\times\R^6)]{
       (D_z\widehat{\cF}_c(z))^{-1}}
 &\le2\norm[\cB(X_{s-1}\times\R^6,X_s\times\R^6)]{
       \widehat{\cL}_\infty^{-1}}.
 \end{aligned}
\]
This proves \eqref{eq:2.29}.
\end{proof}

{
\begin{lemma}
\label{lem:2.14}
For every \(s\ge\tfrac12\),
\begin{equation}
 \norm[H^s]{f_c-f_\infty}
 +\norm[H^s]{h_c-h_\infty}
 +\abs{\lambda_c-\lambda_\infty}
 \lesssim c^{-2}.
 \label{eq:2.31}
\end{equation}
\end{lemma}
}

\begin{proof}
Fix \(s\ge2\).  Since \(\widehat{\cF}_c(z_c)=0\),
\begin{equation*}
 \left[\int_0^1D_z\widehat{\cF}_c
 \bigl(z_\infty+t(z_c-z_\infty)\bigr)\,\dd t\right]
 (z_c-z_\infty)=-\widehat{\cF}_c(z_\infty).
\end{equation*}
By \cref{lem:2.12} and
\eqref{eq:2.30}, the operator 
\[
\int_0^1D_z\widehat{\cF}_c
 \bigl(z_\infty+t(z_c-z_\infty)\bigr)\,\dd t
\]
converges in operator norm to
\(D_z\widehat{\cF}_\infty(z_\infty)\).  It is therefore uniformly
invertible for large \(c\).  Since
\begin{equation*}
 \norm[X_{s-1}\times\R^6]{\widehat{\cF}_c(z_\infty)}
 \lesssim c^{-2},
\end{equation*}
we obtain \eqref{eq:2.31} for \(s\ge2\).

\end{proof}

\section{Uniqueness and nondegeneracy of ground states}
\label{sec:3}

Let
\begin{equation*}
 \cO_\infty:=\cH\cdot(f_\infty,h_\infty)
 \subset H^2(\R^3,\C^2)\times H^2(\R^3,\C^2).
\end{equation*}
{
The orbit of \((f_c,h_c)\) has the local parametrization
\begin{equation*}
 \Theta_c(p)(x)
 :=\binom{
 F_c(\abs{x+y})\,\Sigma(a)}{
 -\ii G_c(\abs{x+y})(\sigma\cdot\widehat{x+y})\,\Sigma(a)},
 \qquad
 p=(y,a)\in\R^3\times\R^3,
\end{equation*}
where
\begin{equation*}
 \Sigma(a)
 :=\frac{(1+\ii a_1)\mathbf e_1+(a_2+\ii a_3)\mathbf e_2}
 {\sqrt{1+\abs a^2}}
 =:\binom{\Sigma_1(a)}{\Sigma_2(a)}\in\mathbb S^3.
\end{equation*}
Define
\begin{equation*}
 S(a):=
 \begin{pmatrix}
  \Sigma_1(a)&-\overline{\Sigma_2(a)}\\
  \Sigma_2(a)& \overline{\Sigma_1(a)}
 \end{pmatrix}\in SU(2),
 \qquad
 \cT_p:=\cT_{(-y,S(a),1)}.
\end{equation*}
Then \(S(a)\mathbf e_1=\Sigma(a)\), and the covariance
\eqref{eq:1.9} gives
\begin{equation*}
 \Theta_c(p)=\cT_p(f_c,h_c).
\end{equation*}
Since \(a\mapsto\Sigma(a)\) is a coordinate chart on \(\mathbb S^3\) near
\(\mathbf e_1\), the map \(p\mapsto\cT_p\) is an explicit smooth local
section of the quotient by the one-dimensional stabilizer.  Define
\begin{equation}
 y_{j,c}:=\partial_{p_j}\Theta_c(p)\big|_{p=0},
 \qquad 1\le j\le6.
 \label{eq:3.1}
\end{equation}
Here \(F_\infty=Q\) and \(G_\infty=Q'/(2m)\).  Direct differentiation gives
\begin{equation}
 (y_{j,\infty})_\uparrow=v_j,
 \qquad 1\le j\le6.
 \label{eq:3.2}
\end{equation}
Thus the upper components of the six tangent vectors form precisely the
fixed basis of \(\cK\).

\begin{lemma}\label{lem:3.1}
There are \(\delta_0,\rho>0\) such that, whenever
\((f,h)\in H^2(\R^3,\C^2)\times H^2(\R^3,\C^2)\),
\(\gamma_0\in\cH\), and
\begin{equation}\label{eq:3.3}
 \norm[H^2]{\cT_{\gamma_0}^{-1}f-f_\infty}
 +\norm[H^2]{\cT_{\gamma_0}^{-1}h-h_\infty}<\delta_0,
\end{equation}
there is a unique \(p\in\R^6\), \(\abs p<\rho\), in the local section above
such that
\begin{equation*}
 \ipR{\cT_p^{-1}\cT_{\gamma_0}^{-1}f-f_\infty}{v_j}=0,
 \qquad 1\le j\le6.
\end{equation*}
\end{lemma}
\begin{proof}
For \(p=(y,a)\), write \(S=S(a)\) and \(R=R_{S(a)}\).  By the
radiality of \(Q\),
\begin{align*}
 \cT_pv_j(x)
 &=\sum_{k=1}^3R_{kj}\,\partial_kQ(x+y)\,\Sigma(a),
 &&1\le j\le3,\\
 \cT_pv_4(x)&=\ii Q(\abs{x+y})\Sigma(a),\\
 \cT_pv_5(x)&=Q(\abs{x+y})S(a)\mathbf e_2,\\
 \cT_pv_6(x)&=\ii Q(\abs{x+y})S(a)\mathbf e_2.
\end{align*}
Since \(Q\in\bigcap\limits_{s\ge0}H^s(\R^3)\), for every multi-index
\(\alpha\) and sufficiently small \(\rho_0>0\),
\begin{equation}
 \sup_{\abs p\le\rho_0}
 \norm[H^2]{\partial_p^\alpha(\cT_pv_i)}<\infty,
 \qquad 1\le i\le6.
 \label{eq:3.4}
\end{equation}

For \(p\in B_{\rho_0}(0)\) and \(w\in H^2(\R^3,\C^2)\), define
\begin{equation*}
 \cP(p,w)
 :=\bigl(\ipR{\cT_p^{-1}w-f_\infty}{v_i}\bigr)_{i=1}^6.
\end{equation*}
The action \(\cT_p\) is unitary on \(L^2(\R^3,\C^2)\), hence
\begin{equation}
 \cP_i(p,w)
 =\ipR{w}{\cT_pv_i}-\ipR{f_\infty}{v_i},
 \qquad 1\le i\le6.
 \label{eq:3.5}
\end{equation}
It follows from \eqref{eq:3.4} that
\begin{equation*}
 \cP\in C^\infty\bigl(
 B_{\rho_0}(0)\times H^2(\R^3,\C^2),\R^6
 \bigr).
\end{equation*}
Moreover,
\begin{equation*}
 \cP(0,f_\infty)=0.
\end{equation*}
For \(1\le i,j\le6\),
\begin{equation*}
 \ipR{\cT_p f_\infty}{\cT_pv_i}
 =\ipR{f_\infty}{v_i}.
\end{equation*}
Differentiating at \(p=0\) in the \(p_j\) direction and using
\eqref{eq:3.2},
\begin{align*}
 \partial_{p_j}\cP_i(0,f_\infty)
 &=\ipR{f_\infty}{
 \partial_{p_j}(\cT_pv_i)|_{p=0}}\\
 &=-\ipR{
 \partial_{p_j}(\cT_pf_\infty)|_{p=0}}{v_i}\\
 &=-\ipR{(y_{j,\infty})_\uparrow}{v_i}\\
 &=-\ipR{v_j}{v_i}
 =-\Gamma_{ij}.
\end{align*}
Thus
\begin{equation}
 \mathrm d_p\cP(0,f_\infty)=-\Gamma.
 \label{eq:3.6}
\end{equation}
Choose \(0<\rho<\rho_0\) and \(\delta_0>0\) such that
\begin{equation}
 \sup_{\substack{\abs p\le\rho\\
 \norm[H^2]{w-f_\infty}\le\delta_0}}
 \norm[\cB(\R^6,\R^6)]{
 \Gamma^{-1}\bigl(\mathrm d_p\cP(p,w)+\Gamma\bigr)}
 \le\frac12
 \label{eq:3.7}
\end{equation}
and
\begin{equation}
       \abs{\Gamma^{-1}\cP(0,w)}\le\frac{\rho}{4}
 \label{eq:3.8}
\end{equation}
provided
\begin{equation*}
 \norm[H^2]{w-f_\infty}\le\delta_0.
\end{equation*}
Indeed,
\begin{align*}
 \abs{\Gamma^{-1}\cP(0,w)}
 &\le
 \norm[\cB(\R^6,\R^6)]{\Gamma^{-1}}
 \left(
 \sum_{i=1}^6
 \abs{\ipR{w-f_\infty}{v_i}}^2
 \right)^{1/2}\\
 &\le
 \norm[\cB(\R^6,\R^6)]{\Gamma^{-1}}
 \left(
 \sum_{i=1}^6\norm[L^2]{v_i}^2
 \right)^{1/2}
 \norm[H^2]{w-f_\infty}.
\end{align*}
For
\[
 \norm[H^2]{w-f_\infty}\le\delta_0,
\]
set
\begin{equation*}
 \Psi_w(p):=p+\Gamma^{-1}\cP(p,w),
 \qquad p\in\overline B_\rho(0).
\end{equation*}
By \eqref{eq:3.7},
\begin{align*}
 \norm[\cB(\R^6,\R^6)]{\mathrm d\Psi_w(p)}
 &=
 \norm[\cB(\R^6,\R^6)]{
 \Gamma^{-1}
 \bigl(\mathrm d_p\cP(p,w)+\Gamma\bigr)}
 \le\frac12,\\
 \abs{\Psi_w(p)-\Psi_w(q)}
 &\le\frac12\abs{p-q}.
\end{align*}
Moreover, by \eqref{eq:3.8},
\begin{equation*}
 \abs{\Psi_w(p)}
 \le
 \abs{\Psi_w(p)-\Psi_w(0)}
 +\abs{\Psi_w(0)}
 \le
 \frac12\abs p+\frac{\rho}{4}
 \le\frac{3\rho}{4}.
\end{equation*}
Hence
\begin{equation*}
 \Psi_w\bigl(\overline B_\rho(0)\bigr)
 \subset\overline B_\rho(0).
\end{equation*}
Banach's fixed-point theorem gives a unique
\(p\in\overline B_\rho(0)\) satisfying
\begin{equation*}
 p=\Psi_w(p)
 \quad\Longleftrightarrow\quad
 \cP(p,w)=0.
\end{equation*}
Furthermore,
\begin{align}
 \abs p
 &\le
 \frac12\abs p+
 \abs{\Gamma^{-1}\cP(0,w)},\notag\\
 \abs p
 &\le
 2\abs{\Gamma^{-1}\cP(0,w)}
 \lesssim\norm[H^2]{w-f_\infty},
 \qquad \abs p<\rho.
 \label{eq:3.9}
\end{align}
Set
\begin{equation*}
 w:=\cT_{\gamma_0}^{-1}f.
\end{equation*}
\eqref{eq:3.3} gives
\begin{equation*}
 \norm[H^2]{w-f_\infty}<\delta_0.
\end{equation*}
Therefore there is a unique \(p\in\R^6\), \(\abs p<\rho\), such that
\begin{equation*}
 \ipR{
 \cT_p^{-1}\cT_{\gamma_0}^{-1}f-f_\infty
 }{v_j}=0,
 \qquad 1\le j\le6.
\end{equation*}
Finally, \(\cT_p\) is an isometry on \(H^2\), and
\begin{equation*}
 \norm[H^2]{\cT_p^{-1}f_\infty-f_\infty}
 +\norm[H^2]{\cT_p^{-1}h_\infty-h_\infty}
 \lesssim\abs p.
\end{equation*}
Together with \eqref{eq:3.9},
\begin{align*}
 &\norm[H^2]{
 \cT_p^{-1}\cT_{\gamma_0}^{-1}f-f_\infty}
 +\norm[H^2]{
 \cT_p^{-1}\cT_{\gamma_0}^{-1}h-h_\infty}\\
 &\qquad\lesssim
 \norm[H^2]{\cT_{\gamma_0}^{-1}f-f_\infty}
 +\norm[H^2]{\cT_{\gamma_0}^{-1}h-h_\infty}.
\end{align*}
\end{proof}

\begin{proposition}\label{prop:3.2}
There are \(\delta>0\) and \(c_2>0\) such that, for \(c>c_2\), every
solution
\((\widetilde f_c,\widetilde h_c,\widetilde\lambda_c)\) of
\eqref{eq:1.2} satisfying
\begin{equation}\label{eq:3.10}
 {\dist_{X_2}}
 \bigl((\widetilde f_c,\widetilde h_c,\widetilde\lambda_c),
       \cO_\infty\times\{\lambda_\infty\}\bigr)<\delta
\end{equation}
has the form
\begin{equation*}
 (\widetilde f_c,\widetilde h_c,\widetilde\lambda_c)
 =(\cT_\gamma f_c,\cT_\gamma h_c,\lambda_c)
 \quad\text{for some }\gamma\in\cH.
\end{equation*}
\end{proposition}
\begin{proof}
Let \(r>0\) be as in \cref{lem:2.10}.  Increase \(c_2\) so
that, for every \(c>c_2\),
\begin{equation*}
 (f_c,h_c,\lambda_c,\bm a_c)
 =(f_c,h_c,\lambda_c,0)\in B_r(z_\infty).
\end{equation*}
By the proof of \cref{lem:3.1}, the corresponding parameter
\(p\) satisfies
\begin{equation*}
 \abs p
 \lesssim
 \norm[H^2]{\cT_{\gamma_0}^{-1}f-f_\infty},
\end{equation*}
and, for \(\abs p<\rho\),
\begin{equation*}
 \norm[H^2]{\cT_p^{-1}f_\infty-f_\infty}
 +\norm[H^2]{\cT_p^{-1}h_\infty-h_\infty}
 \lesssim \abs p.
\end{equation*}
Since \(\cT_p\) is an isometry on \(H^2\),
\begin{align*}
 &\norm[X_2\times\R^6]{
 (\cT_p^{-1}\cT_{\gamma_0}^{-1}f,
  \cT_p^{-1}\cT_{\gamma_0}^{-1}h,
  \lambda,0)-z_\infty}
 \\
 &\quad\lesssim
 \norm[H^2]{\cT_{\gamma_0}^{-1}f-f_\infty}
 +\norm[H^2]{\cT_{\gamma_0}^{-1}h-h_\infty}
 +\abs{\lambda-\lambda_\infty}.
\end{align*}
Thus there is \(C\ge1\), independent of \(c\), such that
\begin{align}
 &\norm[H^2]{\cT_{\gamma_0}^{-1}f-f_\infty}
 +\norm[H^2]{\cT_{\gamma_0}^{-1}h-h_\infty}
 \notag\\
 &\qquad\le
 C\norm[H^2\times H^2\times\R]{
 (\cT_{\gamma_0}^{-1}f-f_\infty,
  \cT_{\gamma_0}^{-1}h-h_\infty,
  \lambda-\lambda_\infty)},
 \label{eq:3.11}\\
 &\norm[X_2\times\R^6]{
 (\cT_p^{-1}\cT_{\gamma_0}^{-1}f,
  \cT_p^{-1}\cT_{\gamma_0}^{-1}h,
  \lambda,0)-z_\infty}
 \notag\\
 &\qquad\le
 C\norm[H^2\times H^2\times\R]{
 (\cT_{\gamma_0}^{-1}f-f_\infty,
  \cT_{\gamma_0}^{-1}h-h_\infty,
  \lambda-\lambda_\infty)}.
 \label{eq:3.12}
\end{align}
Choose \(\delta>0\) such that
\begin{equation*}
 2C\delta<\min\{\delta_0,r\}.
\end{equation*}
Let
\((\widetilde f_c,\widetilde h_c,\widetilde\lambda_c)\)
satisfy \eqref{eq:3.10}.  There is \(\gamma_0\in\cH\) such that
\begin{equation}
 \left\|
 \left(
 \cT_{\gamma_0}^{-1}\widetilde f_c-f_\infty,
 \cT_{\gamma_0}^{-1}\widetilde h_c-h_\infty,
 \widetilde\lambda_c-\lambda_\infty
 \right)
 \right\|_{H^2\times H^2\times\R}
 <2\delta.
 \label{eq:3.13}
\end{equation}
By \eqref{eq:3.11} and
\eqref{eq:3.13},
\begin{equation*}
 \norm[H^2]{\cT_{\gamma_0}^{-1}\widetilde f_c-f_\infty}
 +\norm[H^2]{\cT_{\gamma_0}^{-1}\widetilde h_c-h_\infty}
 <2C\delta<\delta_0.
\end{equation*}
Applying \cref{lem:3.1}, we obtain \(p\in\R^6\),
\(\abs p<\rho\), such that
\begin{equation}
 \ipR{
 \cT_p^{-1}\cT_{\gamma_0}^{-1}\widetilde f_c-f_\infty
 }{v_j}=0,
 \qquad 1\le j\le6.
 \label{eq:3.14}
\end{equation}
Set
\begin{equation*}
 \overline f_c
 :=\cT_p^{-1}\cT_{\gamma_0}^{-1}\widetilde f_c,
 \qquad
 \overline h_c
 :=\cT_p^{-1}\cT_{\gamma_0}^{-1}\widetilde h_c.
\end{equation*}
By \eqref{eq:3.12} and
\eqref{eq:3.13},
\begin{equation}
 \norm[X_2\times\R^6]{
 (\overline f_c,\overline h_c,\widetilde\lambda_c,0)-z_\infty}
 <2C\delta<r.
 \label{eq:3.15}
\end{equation}
For \(\gamma=(y,S,e^{\ii\theta})\in\cH\), it's easy to see
\begin{equation}
 \cF_c(\cT_\gamma f,\cT_\gamma h,\lambda)
 =
 \operatorname{diag}(\cT_\gamma,\cT_\gamma,1)
 \cF_c(f,h,\lambda).
 \label{eq:3.16}
\end{equation}
Since
\begin{equation*}
 \cF_c(\widetilde f_c,\widetilde h_c,
 \widetilde\lambda_c)=0,
\end{equation*}
\eqref{eq:3.16} gives
\begin{equation}
 \cF_c(\overline f_c,\overline h_c,
 \widetilde\lambda_c)=0.
 \label{eq:3.17}
\end{equation}
Moreover, \eqref{eq:3.14} yields
\begin{equation*}
 \ell_j(\overline f_c)
 =\ipR{\overline f_c-f_\infty}{v_j}=0,
 \qquad 1\le j\le6.
\end{equation*}
Combining this with
\eqref{eq:3.17}, we obtain
\begin{equation}
 \widehat{\cF}_c(
 \overline f_c,\overline h_c,
 \widetilde\lambda_c,0)=0.
 \label{eq:3.18}
\end{equation}
By \eqref{eq:3.15},
\eqref{eq:3.18}, and the uniqueness in
\cref{lem:2.10},
\begin{equation*}
 (\overline f_c,\overline h_c,
 \widetilde\lambda_c,0)
 =(f_c,h_c,\lambda_c,\bm a_c).
\end{equation*}
Therefore
\begin{equation*}
 (\widetilde f_c,\widetilde h_c)
 =\cT_{\gamma_0}\cT_p(f_c,h_c).
\end{equation*}
There exists \(\gamma\in\cH\) such that
\begin{equation*}
 \cT_{\gamma_0}\cT_p=\cT_\gamma.
\end{equation*}
Hence
\begin{equation*}
 (\widetilde f_c,\widetilde h_c,\widetilde\lambda_c)
 =
 (\cT_\gamma f_c,\cT_\gamma h_c,\lambda_c).
\end{equation*}
\end{proof}

}

{
\begin{lemma}\label{lem:3.3}
Let \(u=(f,g)^T\in\cS\) be a constrained critical point of \(\cE_c\),
with multiplier \(\omega\) which satisfies
\[|\omega|<mc^2.\] Then
\begin{align}
 \cE_c(u)
 &=mc^2\int_{\R^3}u^\dagger\beta u\,\dd x
 =mc^2\left(\norm[L^2]{f}^2-\norm[L^2]{g}^2\right),
 \label{eq:3.19}\\
 mc^2-\omega
 &=2mc^2\norm[L^2]{g}^2
 +\frac12\int _{\R^3\times\R^3}
 \frac{\abs{u(x)}^2\abs{u(y)}^2}{\abs{x-y}}\,\dd x\dd y.
 \label{eq:3.20}
\end{align}
\end{lemma}

\begin{proof}
By \cref{lem:2.2}, $u\in\bigcap\limits_{s>0}H^s$.
Since $|\omega|<mc^2$, Kato's inequality and an exterior comparison
argument, followed by
interior elliptic estimates, yield
$|u(x)|+|\nabla u(x)|\le Ce^{-a|x|}$ for some $C,a>0$.
Test with \(x\cdot\nabla u\) gives
\begin{align*}
 \Re\ip{\mathscr{D}_cu}{x\cdot\nabla u}
 &=-\ip{\mathscr{D}_cu}{u}-\frac12mc^2\ip{\beta u}{u},\\
 \Re\ip{\omega u}{x\cdot\nabla u}
 &=-\frac32\omega,\\
 2\int_{\R^3}\abs{u(x)}^2x\cdot\nabla
 \bigl(\abs{x}^{-1}\ast \abs u^2\bigr)(x)\,\dd x
 &=-\int _{\R^3\times\R^3}
 \frac{\abs{u(x)}^2\abs{u(y)}^2}{\abs{x-y}}\,\dd x\dd y,\\
 \Re\int_{\R^3}
 \bigl(\abs{x}^{-1}\ast \abs u^2\bigr)u^\dagger x\cdot\nabla u\,\dd x
 &=-\frac54\int _{\R^3\times\R^3}
 \frac{\abs{u(x)}^2\abs{u(y)}^2}{\abs{x-y}}\,\dd x\dd y.
\end{align*}
Hence
\begin{align*}
 \ip{\mathscr{D}_cu}{u}+\frac12mc^2\ip{\beta u}{u}
 =\frac54\int _{\R^3\times\R^3}
 \frac{\abs{u(x)}^2\abs{u(y)}^2}{\abs{x-y}}\,\dd x\dd y
 +\frac32\omega,
\end{align*}
\[
 \ip{\mathscr{D}_cu}{u}
 -\int _{\R^3\times\R^3}
 \frac{\abs{u(x)}^2\abs{u(y)}^2}{\abs{x-y}}\,\dd x\dd y
 =\omega.
\]
Eliminating \(\omega\),
\begin{equation*}
 c\Re\int_{\R^3}u^\dagger(-\ii\alpha\cdot\nabla)u\,\dd x
 =\frac12\int _{\R^3\times\R^3}
 \frac{\abs{u(x)}^2\abs{u(y)}^2}{\abs{x-y}}\,\dd x\dd y.
\end{equation*}
Therefore
\begin{align*}
 \cE_c(u)
 &=mc^2\int_{\R^3}u^\dagger\beta u\,\dd x
 =mc^2\left(\norm[L^2]{f}^2-\norm[L^2]{g}^2\right),\\
 mc^2-\omega
 &=mc^2-mc^2\int_{\R^3}u^\dagger\beta u\,\dd x
 +\frac12\int _{\R^3\times\R^3}
 \frac{\abs{u(x)}^2\abs{u(y)}^2}{\abs{x-y}}\,\dd x\dd y\\
 &=2mc^2\norm[L^2]{g}^2
 +\frac12\int _{\R^3\times\R^3}
 \frac{\abs{u(x)}^2\abs{u(y)}^2}{\abs{x-y}}\,\dd x\dd y.
\end{align*}
\end{proof}

}

\begin{lemma}\label{lem:3.4}
Let \(f_c,h_c\in H^1(\R^3,\C^2)\) satisfy
\begin{equation*}
 \norm[L^2]{f_c}^2+c^{-2}\norm[L^2]{h_c}^2=1.
\end{equation*}
Then
\begin{align}
 \cE_c\left(\binom{f_c}{c^{-1}h_c}\right)-mc^2
 ={}&\cE_\infty(f_c)
 -2m\norm[L^2]{
 h_c-\frac1{2m}\cD f_c}^2
 \notag\\
 &-c^{-2}\int _{\R^3\times\R^3}
 \frac{\abs{f_c(x)}^2\abs{h_c(y)}^2}{\abs{x-y}}\,
 \dd x\dd y
 \notag\\
 &-\frac{c^{-4}}2
 \int _{\R^3\times\R^3}
 \frac{\abs{h_c(x)}^2\abs{h_c(y)}^2}{\abs{x-y}}\,
 \dd x\dd y.
 \label{eq:3.21}
\end{align}
Moreover, for \(u_c=(f_c,g_c)^T\in\cG_c\), set
\begin{equation*}
 h_c:=cg_c,
 \qquad
 \widehat f_c:=\frac{f_c}{\norm[L^2]{f_c}}.
\end{equation*}
Then
\begin{equation}
 \sup_{u_c=(f_c,g_c)^T\in\cG_c}
 \abs{
 \cE_c(u_c)-mc^2-\cE_\infty(\widehat f_c)}
 \lesssim c^{-2}.
 \label{eq:3.22}
\end{equation}
\end{lemma}

\begin{proof}
A direct computation gives
\begin{align*}
 \cE_c\left(\binom{f_c}{c^{-1}h_c}\right)
 ={}&mc^2\norm[L^2]{f_c}^2
 -m\norm[L^2]{h_c}^2
 +2\ipR{\cD f_c}{h_c}\\
 &-\frac12
 \int _{\R^3\times\R^3}
 \frac{
 \bigl(\abs{f_c(x)}^2+c^{-2}\abs{h_c(x)}^2\bigr)
 \bigl(\abs{f_c(y)}^2+c^{-2}\abs{h_c(y)}^2\bigr)}
 {\abs{x-y}}\,
 \dd x\dd y.
\end{align*}
By
\[
 mc^2\left(\norm[L^2]{f_c}^2-1\right)
 =-m\norm[L^2]{h_c}^2
\]
and
\[
 \norm[L^2]{\cD f_c}^2
 =\norm[L^2]{\nabla f_c}^2,
\]
we have
\begin{align*}
 &mc^2\norm[L^2]{f_c}^2
 -m\norm[L^2]{h_c}^2
 -mc^2
 +2\ipR{\cD f_c}{h_c}\\
 &\qquad
 =-2m\norm[L^2]{h_c}^2
 +2\ipR{\cD f_c}{h_c}\\
 &\qquad
 =\frac1{2m}\norm[L^2]{\cD f_c}^2
 -2m\norm[L^2]{
 h_c-\frac1{2m}\cD f_c}^2\\
 &\qquad
 =\frac1{2m}\norm[L^2]{\nabla f_c}^2
 -2m\norm[L^2]{
 h_c-\frac1{2m}\cD f_c}^2.
\end{align*}
Moreover,
\begin{align*}
 &-\frac12
 \int _{\R^3\times\R^3}
 \frac{
 \bigl(\abs{f_c(x)}^2+c^{-2}\abs{h_c(x)}^2\bigr)
 \bigl(\abs{f_c(y)}^2+c^{-2}\abs{h_c(y)}^2\bigr)}
 {\abs{x-y}}\,
 \dd x\dd y\\
 &\quad=
 -\frac12
 \int _{\R^3\times\R^3}
 \frac{\abs{f_c(x)}^2\abs{f_c(y)}^2}{\abs{x-y}}\,
 \dd x\dd y\\
 &\qquad
 -c^{-2}
 \int _{\R^3\times\R^3}
 \frac{\abs{f_c(x)}^2\abs{h_c(y)}^2}{\abs{x-y}}\,
 \dd x\dd y\\
 &\qquad
 -\frac{c^{-4}}2
 \int _{\R^3\times\R^3}
 \frac{\abs{h_c(x)}^2\abs{h_c(y)}^2}{\abs{x-y}}\,
 \dd x\dd y.
\end{align*}
Hence \eqref{eq:3.21} follows.

The second
equation in \eqref{eq:1.2} gives
\begin{equation}
 h_c-\frac1{2m}\cD f_c
 =
 \frac{c^{-2}}{2m}
 \left\{
 \lambda_c-
 \abs{x}^{-1}\ast
 \left(\abs{f_c}^2+c^{-2}\abs{h_c}^2\right)
 \right\}h_c.
 \label{eq:3.23}
\end{equation}
By
\eqref{eq:2.1},
\begin{align*}
 &\norm[L^2]{
 h_c-\frac1{2m}\cD f_c}\\
 &\quad\le
 \frac{c^{-2}}{2m}
 \left[
 \lambda_c\norm[L^2]{h_c}
 +
 \norm[L^2]{
 \left(
 \abs{x}^{-1}\ast \abs{f_c}^2
 \right)h_c}
 +
 c^{-2}\norm[L^2]{
 \left(
 \abs{x}^{-1}\ast \abs{h_c}^2
 \right)h_c}
 \right]\\
 &\quad\lesssim c^{-2}.
\end{align*}
Therefore
\begin{equation}
 \sup_{u_c\in\cG_c}
 \norm[L^2]{
 h_c-\frac1{2m}\cD f_c}^2
 \lesssim c^{-4}.
 \label{eq:3.24}
\end{equation}
It follows from \eqref{eq:3.21} and
\eqref{eq:3.24} that
\begin{equation}
 \sup_{u_c\in\cG_c}
 \abs{
 \cE_c(u_c)-mc^2-\cE_\infty(f_c)}
 \lesssim c^{-2}.
 \label{eq:3.25}
\end{equation}
It's easy to see
\begin{equation*}
 \norm[L^2]{f_c}^2
 =1-c^{-2}\norm[L^2]{h_c}^2
 =1+O(c^{-2}),
\end{equation*}
uniformly for \(u_c\in\cG_c\).  Thus, for all sufficiently large \(c\),
\begin{equation*}
 \norm[L^2]{f_c}^2\ge\frac12,
 \qquad
 \abs{\norm[L^2]{f_c}^{-2}-1}
 +\abs{\norm[L^2]{f_c}^{-4}-1}
 \lesssim c^{-2}.
\end{equation*}
Moreover,
\begin{align*}
 \cE_\infty(\widehat f_c)-\cE_\infty(f_c)
 ={}&
 \frac{\norm[L^2]{f_c}^{-2}-1}{2m}
 \norm[L^2]{\nabla f_c}^2\\
 &-\frac{\norm[L^2]{f_c}^{-4}-1}{2}
 \int _{\R^3\times\R^3}
 \frac{\abs{f_c(x)}^2\abs{f_c(y)}^2}{\abs{x-y}}\,
 \dd x\dd y.
\end{align*}
and hence
\begin{equation}
 \sup_{u_c\in\cG_c}
 \abs{
 \cE_\infty(\widehat f_c)-\cE_\infty(f_c)}
 \lesssim c^{-2}.
 \label{eq:3.26}
\end{equation}
Combining \eqref{eq:3.25} and
\eqref{eq:3.26} gives
\eqref{eq:3.22}.
\end{proof}

{
{The next proposition follows from} \cref{prop:2.1} and \cref{lem:2.4}.
}

\begin{proposition}
\label{prop:3.5}
For each \(u_c=(f_c,g_c)^T\in\cG_c\), let \(\omega_c\) be its multiplier.
Then
\begin{equation*}
 \begin{aligned}
 \sup_{u_c=(f_c,g_c)^T\in\cG_c}\inf_{\gamma\in\cH}
 \bigl(&\norm[H^2]{f_c-\cT_\gamma f_\infty}
 +\norm[H^2]{cg_c-\cT_\gamma h_\infty}\\
 &+\abs{mc^2-\omega_c-\lambda_\infty}\bigr)
 \longrightarrow0
 \qquad\text{as }c\to\infty.
 \end{aligned}
\end{equation*}
\end{proposition}

{

}
By \cref{prop:3.5} and \cref{prop:3.2},
\begin{equation*}
 \cG_c\subset\{\cT_\gamma u_c:\gamma\in\cH\}.
\end{equation*}
For \(\widetilde u_c\in\cG_c\),
\begin{equation*}
 \widetilde u_c=\cT_{\gamma_c}u_c
 \quad\text{for some }\gamma_c\in\cH.
\end{equation*}
The invariance of \eqref{eq:1.1}, and \eqref{eq:1.4}
gives
\begin{equation}
 u_c\in\cG_c,
 \qquad
 \cG_c=\{\cT_\gamma u_c:\gamma\in\cH\}.
 \label{eq:3.27}
\end{equation}
{

\begin{lemma}\label{lem:3.6}
For all sufficiently large $c$, the radial profiles in
\cref{thm:1.1} satisfy
\begin{equation*}
 F_c(r)>0\quad(r\ge0),
 \qquad
 G_c(0)=0,
 \qquad
 G_c(r)<0\quad(r>0).
\end{equation*}
Moreover, $(F_c,G_c)$ is the unique pair of real radial profiles of a
ground state satisfying $F_c>0$.
\end{lemma}

\begin{proof}
Set
\begin{equation*}
 V_c(r):=
 \left[\abs{x}^{-1}\ast
 \left(F_c^2+c^{-2}G_c^2\right)\right](r).
\end{equation*}
By \cref{lem:2.12},
\begin{equation*}
 F_c,G_c,V_c\in C^1([0,\infty)),
 \qquad
 F_c'(0)=0.
\end{equation*}
The second equation in \eqref{eq:2.25} and
$\lambda_c=mc^2-\omega_c$ give
\begin{equation}
 \left(m+c^{-2}\omega_c+c^{-2}V_c(r)\right)G_c(r)=F_c'(r),
 \qquad
 m+c^{-2}\omega_c+c^{-2}V_c(r)>0.
 \label{eq:3.28}
\end{equation}
Hence
\begin{equation*}
 G_c(0)=0.
\end{equation*}
Choose $R>0$ such that
\begin{equation*}
 R^{-1}<\frac{\lambda_\infty}{4}.
\end{equation*}
By \eqref{eq:2.31} and
$H^2(\R^3)\hookrightarrow L^\infty(\R^3)$,
\begin{equation*}
 \norm[L^\infty]{F_c-Q}\longrightarrow0,
 \qquad
 \lambda_c\longrightarrow\lambda_\infty.
\end{equation*}
Thus, for all sufficiently large $c$,
\begin{equation}
 F_c(r)>0\quad(0\le r\le R),
 \qquad
 \lambda_c>\frac{\lambda_\infty}{2}.
 \label{eq:3.29}
\end{equation}
Newton's formula and the mass constraint yield
\begin{align}
 V_c(r)
 &=4\pi\left
 \{\frac1r\int_0^r
 \left(F_c(s)^2+c^{-2}G_c(s)^2\right)s^2\,\dd s
 \right.\notag\\
 &\hspace{7em}\left.
 +\int_r^\infty
 \left(F_c(s)^2+c^{-2}G_c(s)^2\right)s\,\dd s
 \right\}
 \le\frac1r.
 \label{eq:3.30}
\end{align}
Consequently, for $r\ge R$,
\begin{equation}
 \lambda_c-V_c(r)\ge\frac{\lambda_\infty}{4},
 \qquad
 m+c^{-2}\omega_c+c^{-2}V_c(r)\ge m.
 \label{eq:3.31}
\end{equation}

Assume that $F_c(r_0)=0$ for some $r_0\ge R$.  Since
$F_c,G_c\in L^2((0,\infty),r^2\,\dd r)$, there is a sequence
$R_n\to\infty$, $R_n>r_0$, such that
\begin{equation*}
 R_n^2\left(F_c(R_n)^2+G_c(R_n)^2\right)\longrightarrow0.
\end{equation*}
The first two equations in \eqref{eq:2.25} give
\begin{align*}
 &\int_{r_0}^{R_n}
 \left\{
 \bigl(\lambda_c-V_c(r)\bigr)F_c(r)^2
 +\bigl(m+c^{-2}\omega_c+c^{-2}V_c(r)\bigr)G_c(r)^2
 \right\}r^2\,\dd r\\
 &\qquad=
 \left[r^2F_c(r)G_c(r)\right]_{r_0}^{R_n}.
\end{align*}
Letting $n\to\infty$ and using
\eqref{eq:3.31},
\begin{equation*}
 F_c(r)=G_c(r)=0,
 \qquad r\ge r_0.
\end{equation*}
$(F_c, G_c)$ satisfies
\begin{equation}
 \frac{\dd}{\dd r}
 \binom{F_c}{G_c}
 =
 \begin{pmatrix}
  0&m+c^{-2}\omega_c+c^{-2}V_c(r)\\
  \lambda_c-V_c(r)&-2/r
 \end{pmatrix}
 \binom{F_c}{G_c},
 \qquad r>0.
 \label{eq:3.32}
\end{equation}
The coefficient matrix is continuous on 
$(0,\infty)$.  Since
\begin{equation*}
 F_c(r_0)=G_c(r_0)=0,
\end{equation*}
uniqueness for \eqref{eq:3.32} gives
\begin{equation*}
 F_c(r)=G_c(r)=0,
 \qquad r>0,
\end{equation*}
contrary to the mass constraint.  Hence
\begin{equation}
 F_c(r)>0,
 \qquad r\ge0.
 \label{eq:3.33}
\end{equation}
By \eqref{eq:3.30},
\begin{equation*}
 V_c(r)\longrightarrow0
 \qquad(r\to\infty),
\end{equation*}
and 
\begin{equation}
 V_c'(r)
 =-\frac{4\pi}{r^2}\int_0^r
 \left(F_c(s)^2+c^{-2}G_c(s)^2\right)s^2\,\dd s<0,
 \qquad r>0.
 \label{eq:3.34}
\end{equation}
If $V_c(0)\le\lambda_c$, then
\begin{equation*}
 \lambda_c-V_c(r)>0,
 \qquad r>0,
\end{equation*}
and the first equation in \eqref{eq:2.25} gives
\begin{equation*}
 r^2G_c(r)
 =\int_0^r s^2\bigl(\lambda_c-V_c(s)\bigr)F_c(s)\,\dd s>0.
\end{equation*}
By \eqref{eq:3.28},
\begin{equation*}
 F_c'(r)>0,
 \qquad r>0,
\end{equation*}
which contradicts $F_c\in L^2((0,\infty),r^2\,\dd r)$.  Therefore
\begin{equation*}
 V_c(0)>\lambda_c.
\end{equation*}
By \eqref{eq:3.34}, there is a unique $r_*>0$
such that
\begin{equation*}
 V_c(r_*)=\lambda_c,
 \qquad
 \lambda_c-V_c(r)
 \begin{cases}
  <0,&0<r<r_*,\\
  >0,&r>r_*.
 \end{cases}
\end{equation*}
Hence, for $0<r\le r_*$,
\begin{equation}
 r^2G_c(r)
 =\int_0^r s^2\bigl(\lambda_c-V_c(s)\bigr)F_c(s)\,\dd s<0.
 \label{eq:3.35}
\end{equation}
If $G_c$ vanishes in $(r_*,\infty)$, let $r_0>r_*$ be its first zero.
Then
\begin{equation*}
 \bigl(r^2G_c(r)\bigr)'
 =r^2\bigl(\lambda_c-V_c(r)\bigr)F_c(r)>0,
 \qquad r>r_*,
\end{equation*}
so
\begin{equation*}
 G_c(r)>0,
 \qquad r>r_0.
\end{equation*}
 \eqref{eq:3.28} yields
\begin{equation*}
 F_c'(r)>0,
 \qquad r>r_0,
\end{equation*}
again contradicting $F_c\in L^2((0,\infty),r^2\,\dd r)$.  Thus
\begin{equation}
 G_c(r)<0,
 \qquad
 F_c'(r)<0,
 \qquad r>0.
 \label{eq:3.36}
\end{equation}

Let $\widetilde u_c=(\widetilde f_c,\widetilde g_c)^T\in\cG_c$ have
real radial profiles $\widetilde F_c,\widetilde G_c$ with
$\widetilde F_c>0$.  By \eqref{eq:3.27},
\begin{equation*}
 \widetilde u_c
 =\cT_{(y,S,e^{\ii\theta})}u_c
 \qquad\text{for some }(y,S,e^{\ii\theta})\in\cH.
\end{equation*}
Then
\begin{equation}
 \widetilde F_c(\abs x)=F_c(\abs{x-y}).
 \label{eq:3.37}
\end{equation}
If $y\ne0$, choose $r>\abs y$ and set $x=\pm r y/\abs y$.  Then
\eqref{eq:3.37} gives
\begin{equation*}
 F_c(r-\abs y)=F_c(r+\abs y),
\end{equation*}
contrary to \eqref{eq:3.36}.  Hence $y=0$, and
\begin{equation*}
 \widetilde F_c(r)\mathbf e_1
 =F_c(r)e^{\ii\theta}S\mathbf e_1.
\end{equation*}
Since $F_c,\widetilde F_c>0$,
\begin{equation*}
 \widetilde F_c=F_c,
 \qquad
 e^{\ii\theta}S\mathbf e_1=\mathbf e_1.
\end{equation*}
Using \eqref{eq:1.9},
\begin{align*}
 -\ii\widetilde G_c(\abs x)(\sigma\cdot\widehat x)\mathbf e_1
 &=e^{\ii\theta}S
 \left[-\ii G_c(\abs x)
 (\sigma\cdot R_S^{-1}\widehat x)\mathbf e_1\right]\\
 &=-\ii G_c(\abs x)(\sigma\cdot\widehat x)
 e^{\ii\theta}S\mathbf e_1,
\end{align*}
whence
\begin{equation*}
 \widetilde G_c=G_c,
 \qquad
 \widetilde u_c=u_c.
\end{equation*}
Together with \eqref{eq:3.27}, this completes the proof of
\cref{thm:1.1}.
\end{proof}
}

\begin{proposition}
\label{prop:3.7}
For all sufficiently large \(c\), \(\Ker\cL_c\) is the six-dimensional
real space in \eqref{eq:1.11}.
\end{proposition}

\begin{proof}
Set
\begin{equation*}
 \widehat{\cL}_c
 :=D_z\widehat{\cF}_c(f_c,h_c,\lambda_c,0).
\end{equation*}
By \eqref{eq:2.15} and
\eqref{eq:2.31},
\begin{align*}
 (f_c,h_c,\lambda_c,0)&\longrightarrow z_\infty
 &&\text{in }X_2\times\R^6,\\
 \norm[\cB(X_2\times\R^6,X_1\times\R^6)]{
 \widehat{\cL}_c-\widehat{\cL}_\infty}&\longrightarrow0.
\end{align*}
Hence, for all sufficiently large \(c\),
\begin{equation*}
 \norm[\cB(X_2\times\R^6)]{
 \widehat{\cL}_\infty^{-1}
 (\widehat{\cL}_c-\widehat{\cL}_\infty)}<1,
\end{equation*}
and
\begin{equation}
 \widehat{\cL}_c:X_2\times\R^6
 \longrightarrow X_1\times\R^6
 \label{eq:3.38}
\end{equation}
is an isomorphism.
For \((\varphi,\chi,\mu)\in
H^2(\R^3,\C^2)\times H^2(\R^3,\C^2)\times\R\),
\begin{equation}
 \widehat{\cL}_c(\varphi,c\chi,\mu,0)
 =\left(
 \diag(1,c^{-1},1)\cL_c(\varphi,\chi,\mu),
 (\ipR{\varphi}{v_j})_{j=1}^6
 \right),
 \label{eq:3.39}
\end{equation}
Let \(y_{j,c}=(\eta_{j,c},\zeta_{j,c})\) be the tangent vectors in
\eqref{eq:3.1}.  Then
\begin{equation}
 \cL_c(\eta_{j,c},c^{-1}\zeta_{j,c},0)=0,
 \qquad 1\le j\le6.
 \label{eq:3.40}
\end{equation}
Moreover, for $s>0$
\begin{align*}
 (f_c,h_c)&\longrightarrow(f_\infty,h_\infty)
 &&\text{in }H^s\times H^s,\\
 y_{j,c}&\longrightarrow y_{j,\infty}
 &&\text{in }H^s\times H^s,\\
 (M_c)_{ij}:=\ipR{\eta_{j,c}}{v_i}
 &\longrightarrow \ipR{(y_{j,\infty})_\uparrow}{v_i},
 &&1\le i,j\le6.
\end{align*}
Hence, by \eqref{eq:3.2},
\begin{equation*}
 \det M_c\ne0
 \qquad\text{for all sufficiently large }c.
\end{equation*}

Take \((\varphi,\chi,\mu)\in\Ker\cL_c\), and let
\(\bm t=(t_1,\ldots,t_6)\) solve
\begin{equation*}
 M_c\bm t=(\ipR{\varphi}{v_i})_{i=1}^6.
\end{equation*}
Then
\begin{equation*}
 \widetilde\varphi
 :=\varphi-\sum_{j=1}^6t_j\eta_{j,c},
 \qquad
 \widetilde\psi
 :=c\chi-\sum_{j=1}^6t_j\zeta_{j,c},
 \qquad
 \ipR{\widetilde\varphi}{v_i}=0.
\end{equation*}
 \eqref{eq:3.39} and
\eqref{eq:3.40} give
\begin{equation*}
 \widehat{\cL}_c
 (\widetilde\varphi,\widetilde\psi,\mu,0)=0.
\end{equation*}
By \eqref{eq:3.38},
\begin{equation*}
 \widetilde\varphi=\widetilde\psi=0,
 \qquad
 \mu=0,
\end{equation*}
and therefore
\begin{equation*}
 (\varphi,\chi,\mu)
 =\sum_{j=1}^6t_j
 (\eta_{j,c},c^{-1}\zeta_{j,c},0).
\end{equation*}
Together with \eqref{eq:3.40}, this proves
\eqref{eq:1.11}.
\end{proof}

{
\begin{lemma}
\label{lem:3.8}
There exists \(\kappa_\infty>0\) such that
\begin{equation}
 \ipR{\cL\varphi}{\varphi}
 \ge\kappa_\infty\norm[H^1]{\varphi}^2,
 \qquad
 \varphi\in H^1(\R^3,\C^2)\cap\{f_\infty\}^\perp\cap \cK^\perp,
 \label{eq:3.41}
\end{equation}

\end{lemma}

\begin{proof}
For \(\ipR{f_\infty}{\varphi}=0\), constrained minimality gives
\[
 0\le\frac12\left.\frac{\mathrm d^2}{\mathrm dt^2}\right|_{t=0}
 \cE_\infty\!\left(
 \frac{f_\infty+t\varphi}{\norm[L^2]{f_\infty+t\varphi}}\right)
 =\ipR{\cL\varphi}{\varphi}.
\]
If $\ipR{\cL\varphi}{\varphi}=0$, then
\[
 \cL\varphi=\mu f_\infty\quad\text{in }H^{-1}.
\]
for some $\mu\in \R$.
For \(\varphi\perp\cK\), elliptic regularity and \eqref{eq:2.10} give
\[
 \begin{gathered}
 \varphi\in H^2\cap\cK^\perp,\qquad
 \varphi=\mu(\restr{\cL}{\cK^\perp})^{-1}f_\infty,
 \end{gathered}
\]
and
\[
0=\ipR{f_\infty}{\varphi}
   =-\frac{\mu}{4\lambda_\infty},
\]
hence $\mu=0$, $\varphi=0$.
If \eqref{eq:3.41} failed, choose
\[
 \norm[H^1]{\varphi_n}=1,\qquad
 \varphi_n\perp f_\infty,\cK,\qquad
 \ipR{\cL\varphi_n}{\varphi_n}\to0,
 \qquad\varphi_n\rightharpoonup\varphi\text{ in }H^1.
\]
By Rellich compactness embedding and weak lower semicontinuity,
\[
 0\le\ipR{\cL\varphi}{\varphi}
 \le\liminf_n\ipR{\cL\varphi_n}{\varphi_n}=0,
 \qquad\varphi=0.
\]
Hence
\[
 o(1)=\ipR{\cL\varphi_n}{\varphi_n}
 =\frac1{2m}\norm[L^2]{\nabla\varphi_n}^2
  +\lambda_\infty\norm[L^2]{\varphi_n}^2+o(1)
 \ge\min\{(2m)^{-1},\lambda_\infty\}+o(1),
\]
a contradiction.
\end{proof}

{

\begin{lemma}
\label{lem:3.9}
There exist \(c_0>0\) and \(\kappa>0\) such that, for \(c\ge c_0\),
\[
 \mathfrak q_c(u)\ge\kappa\norm[c,+]{u}^2,
 \qquad
 u\in E_c^+\cap\{u_c\}^\perp
 \cap\bigl(\mathrm T_{u_c}(\cH\cdot u_c)\bigr)^\perp.
\]
\end{lemma}

\begin{proof}
Set \(t_c(\xi ):=\sqrt{c^2|\xi |^2+m^2c^4}-mc^2\).
If the results fails, there are \(c_n\to\infty\) and
\(u_n=(\varphi_n,\chi_n)^T\) satisfying
\[
 \begin{gathered}
 u_n\in E_{c_n}^+,
 \qquad
 u_n\perp u_{c_n},\
 \mathrm T_{u_{c_n}}(\cH\cdot u_{c_n}),\\
 \norm[c_n,+]{u_n}^2
 =\norm[L^2]{u_n}^2+
 \int_{\R^3}t_{c_n}(\xi )|\widehat u_n(\xi )|^2\,\dd\xi =1,
 \qquad
 \mathfrak q_{c_n}(u_n)\le n^{-1}.
 \end{gathered}
\]
It's easy to see \(\sup\limits_n\norm[H^{1/2}]{u_n}<\infty\). Moreover,
\[
 \norm[L^2]{\chi_n}^2
 =\int_{\R^3}
 \frac{t_{c_n}(\xi )|\widehat u_n(\xi )|^2}
 {2(t_{c_n}(\xi )+mc_n^2)}\,\dd\xi 
 \le\frac1{2mc_n^2}.
\]
Up to a subsequence,
\[
 u_n\rightharpoonup\binom\varphi0
 \quad\text{in }H^{1/2},
 \qquad
 u_n\to\binom\varphi0
 \quad\text{in }L^2_{\mathrm{loc}}.
\]
For every \(R>0\),
\[
 \begin{gathered}
 \sup_{|\xi |\le R}
 \left|t_{c_n}(\xi )-\frac{|\xi |^2}{2m}\right|\to0
 \end{gathered}
\]
 gives \(\varphi\in H^1\).
By Lemma~\ref*{lem:2.14} and \eqref{eq:3.1},
\[
 \begin{gathered}
 u_{c_n}\to\binom{f_\infty}0,
 \qquad
 \binom{\eta_{j,c_n}}{c_n^{-1}\xi _{j,c_n}}
 \to\binom{v_j}0
 \quad\text{in }H^s,\\
 \ipR{f_\infty}{\varphi}=0,
 \qquad
 \ipR{v_j}{\varphi}=0
 \quad(1\le j\le6).
 \end{gathered}
\]
and hence
\[
 \begin{aligned}
 \norm[L^\infty]{|x|^{-1}\ast(|u_{c_n}|^2-Q^2)}
 \to0.
 \end{aligned}
\]
Lemma~\ref*{lem:3.8} gives
\[
 0\ge\liminf_n\mathfrak q_{c_n}(u_n)
 \ge\ipR{\cL\varphi}{\varphi}
 \ge\kappa_\infty\norm[H^1]{\varphi}^2,
 \qquad
 \varphi=0.
\]
Consequently,
\[
 \begin{aligned}
 n^{-1}&\ge\mathfrak q_{c_n}(u_n)\\
 &=\int_{\R^3}t_{c_n}(\xi )
       |\widehat u_n(\xi )|^2\,\dd\xi 
   +\lambda_{c_n}\norm[L^2]{u_n}^2+o(1)\\
 &\ge\min\{1,\lambda_\infty/2\}+o(1),
 \end{aligned}
\]
a contradiction.
\end{proof}

Proposition~\ref*{prop:3.7} and Lemma~\ref*{lem:3.9}
prove Theorem~\ref*{thm:1.2}.

\begin{lemma}
\label{lem:3.10}
For all sufficiently large \(c\), there holds
\[
  \mathrm d^2\mathcal J_c(w_c)[u,u]
  \ge\kappa\norm[c,+]{u}^2,
  \qquad
  u\in E_c^+\cap\{w_c\}^\perp
  \cap\bigl(\mathrm T_{w_c}(\cH\cdot w_c)\bigr)^\perp,
\]
with \(\kappa>0\) independent of \(c\).
There is \(\delta>0\), independent of \(c\), such that
\[
  \mathcal J_c(w)-e_c
  \asymp
  \inf_{\gamma\in\cH}\norm[c,+]{w-\cT_\gamma w_c}^2
\]
whenever \(w\in E_c^+\), \(\norm[L^2]{w}=1\), and
\[
  \inf_{\gamma\in\cH}\norm[c,+]{w-\cT_\gamma w_c}<\delta.
\]
\end{lemma}

\begin{proof}
The reduction in \cite{ChenDingGuo2026} and
\cite[Proposition~3.6, Lemma~4.5]{Nolasco2021}
gives smoothness of \(\mathcal J_c\) and \(\mathcal J_c(w_c)=e_c\).

For \(v\in\cS\), put \(\eta:=v-u_c\).
\(2\ipR{u_c}{\eta}=-\norm[L^2]{\eta}^2\) give
\[
  \begin{aligned}
  \cE_c(v)-e_c
  &=\ipR{(\mathscr D_c-\omega_c)\eta}{\eta}
    -\int_{\R^3}(|x|^{-1}\ast|u_c|^2)|\eta|^2\,\dd x\\
  &\quad-\frac12\int _{\R^3\times\R^3}
  \frac{(|v(x)|^2-|u_c(x)|^2)(|v(y)|^2-|u_c(y)|^2)}
  {|x-y|}\,\dd x\dd y\\
  &\le\ipR{(\mathscr D_c-\omega_c)\eta}{\eta}.
  \end{aligned}
\]
If
\(P_c^+v\in\Span\{w\}\), fix the phase so that
\(P_c^+v=\norm[L^2]{P_c^+v}\,w\). Since
\(\norm[L^2]{P_c^+u_c}\to1\) and
\(\norm[c,+]{w_c}+|\lambda_c|\le C\),
\[
  \begin{aligned}
  \left|\norm[L^2]{P_c^+v}-\norm[L^2]{P_c^+u_c}\right|
  &=\frac{\left|\norm[L^2]{P_c^-v}^2
               -\norm[L^2]{P_c^-u_c}^2\right|}
           {\norm[L^2]{P_c^+v}+\norm[L^2]{P_c^+u_c}}
  \lesssim\norm[L^2]{P_c^-\eta},\\
  \norm[c,+]{P_c^+\eta}
  &\le\norm[c,+]{w-w_c}+C\norm[L^2]{P_c^-\eta},\\
  \cE_c(v)-e_c
  &\le C\norm[c,+]{P_c^+\eta}^2
       -(mc^2+\omega_c)\norm[L^2]{P_c^-\eta}^2\\
  &\le C\norm[c,+]{w-w_c}^2.
  \end{aligned}
\]
Taking the fibre supremum and using invariance yields
\[
  0\le\mathcal J_c(w)-e_c
  \le C\inf_{\gamma\in\cH}\norm[c,+]{w-\cT_\gamma w_c}^2.
\]

The argument of Lemma~\ref*{lem:3.1}, with the
uniformly invertible orbit Gram matrix, gives
\[
  \begin{gathered}
  \cT_\gamma^{-1}w
  =\frac{w_c+u}{\sqrt{1+\norm[L^2]{u}^2}},\qquad
  u\in E_c^+\cap\{w_c\}^\perp
  \cap\bigl(\mathrm T_{w_c}(\cH\cdot w_c)\bigr)^\perp,\\
  \norm[c,+]{u}
  \asymp\inf_{\gamma\in\cH}\norm[c,+]{w-\cT_\gamma w_c},\qquad
  \mathrm T_{w_c}(\cH\cdot w_c)
  =\frac{P_c^+\mathrm T_{u_c}(\cH\cdot u_c)}
         {\norm[L^2]{P_c^+u_c}}.
  \end{gathered}
\]
Fixed negative component, Lemmas~\ref*{lem:2.2}
and~\ref*{lem:3.9} imply, uniformly for large \(c\),
\[
  \begin{aligned}
  \mathcal J_c(w)-e_c
  &\ge\cE_c\!\left(
  \norm[L^2]{P_c^+u_c}
  \frac{w_c+u}{\sqrt{1+\norm[L^2]{u}^2}}+P_c^-u_c
  \right)-e_c\\
  &=\norm[L^2]{P_c^+u_c}^2\mathfrak q_c(u)
    +O(\norm[c,+]{u}^3)\\
  &\gtrsim\norm[c,+]{u}^2
  \asymp\inf_{\gamma\in\cH}\norm[c,+]{w-\cT_\gamma w_c}^2
  \end{aligned}
\]
for sufficiently small \(\norm[c,+]{u}\). The same comparison
with \(tu\), differentiated at \(t=0\), gives
\[
  \mathrm d^2\mathcal J_c(w_c)[u,u]
  \ge2\norm[L^2]{P_c^+u_c}^2\mathfrak q_c(u)
  \gtrsim\norm[c,+]{u}^2.
  \qedhere
\]
\end{proof}
}

} 

\section{Asymptotic expansions of ground state}
\label{sec:4}
\setlength{\jot}{3pt}

In this section, we
set $\varepsilon=c^{-2}$ and define
\begin{equation}
 \widehat{\cF}(\varepsilon,z)
 :=\widehat{\cF}_\infty(z)+\varepsilon\cC_1(z)
                  +\varepsilon^2\cC_2(z),
 \qquad \widehat{\cF}(c^{-2},z)=\widehat{\cF}_c(z),
 \label{eq:4.1}
\end{equation}
where, for $z=(f,h,\lambda,\bm a)$,
\begin{equation}
 \cC_1(z)=\begin{pmatrix}
 -(|x|^{-1}\ast|h|^2)f\\[0.3em]
 (\lambda-|x|^{-1}\ast|f|^2)h\\[0.3em]
 \norm[L^2]{h}^2\\[0.2em]0
 \end{pmatrix},
 \qquad
 \cC_2(z)=\begin{pmatrix}
 0\\[0.3em]-(|x|^{-1}\ast|h|^2)h\\[0.3em]0\\[0.2em]0
 \end{pmatrix}.
 \label{eq:4.2}
\end{equation}
The maps $\widehat{\cF}_\infty,\cC_1,\cC_2$ are real polynomials in $z$ of degree at most three.

\begin{lemma}\label{lem:4.1}
There is a unique
\[
 z_1=(f_1,h_1,\lambda_1,0)
 \in\bigcap_{s\ge0}
 \bigl(\cW^s_{\mathrm{up}}\times\cW^s_{\mathrm{down}}
                     \times\R\times\{0\}\bigr)
\]
satisfying
\begin{equation}
 \widehat{\cL}_\infty z_1
 =\begin{pmatrix}
 (|x|^{-1}\ast|h_\infty|^2)f_\infty\\[0.3em]
 -(\lambda_\infty-|x|^{-1}\ast Q^2)h_\infty\\[0.3em]
 -\norm[L^2]{h_\infty}^2\\[0.2em]0
 \end{pmatrix}.
 \label{eq:4.3}
\end{equation}
It is also the unique solution of \eqref{eq:4.3} in
$X_s\times\R^6$ for each $s\ge0$. For every $s\ge0$,
\begin{equation}
 \norm[X_s\times\R^6]{c^2(z_c-z_\infty)-z_1}
 \le C_sc^{-2}\longrightarrow0.
 \label{eq:4.4}
\end{equation}
\end{lemma}

\begin{proof}
The right-hand side is $-\cC_1(z_\infty)$ and belongs to
$\cW^{s-1}_{\mathrm{up}}\times\cW^{s-1}_{\mathrm{down}}
\times\R\times\{0\}$ for every $s\ge2$.
Lemma~\ref{lem:2.9} gives $z_1$.

Fix $s\ge2$. Lemma~\ref{lem:2.14} gives
$\norm[X_s\times\R^6]{z_c-z_\infty}\le C_sc^{-2}$.
The polynomial identity \eqref{eq:4.1} yields
\[
 \begin{aligned}
 &\norm[X_{s-1}\times\R^6]{
 \widehat{\cF}_c(z_c)-\widehat{\cF}_c(z_\infty)
                  -\widehat{\cL}_\infty(z_c-z_\infty)}\\
 &\qquad\le C_s\bigl(c^{-2}
            +\norm[X_s\times\R^6]{z_c-z_\infty}\bigr)
                  \norm[X_s\times\R^6]{z_c-z_\infty}
 \le C_sc^{-4}.
 \end{aligned}
\]
Since
\[
 \widehat{\cF}_c(z_c)=0,
 \qquad
 \widehat{\cF}_c(z_\infty)
 =c^{-2}\cC_1(z_\infty)+c^{-4}\cC_2(z_\infty),
\]
we obtain
\[
 \begin{aligned}
 \norm[X_s\times\R^6]{z_c-z_\infty-c^{-2}z_1}
 &\le C_s\norm[X_{s-1}\times\R^6]{
       \widehat{\cL}_\infty(z_c-z_\infty-c^{-2}z_1)}\\
 &\le C_sc^{-4}.
 \end{aligned}
\]

\end{proof}

\begin{lemma}\label{lem:4.2}
There exist unique 
\[
 \begin{gathered}
 z_n=(f_n,h_n,\lambda_n,0)
 \in\bigcap_{s\ge0}
 \bigl(\cW^s_{\mathrm{up}}\times\cW^s_{\mathrm{down}}
       \times\R\times\{0\}\bigr),\\
 \mathfrak e_n\in\R,\qquad n\ge1,
 \end{gathered}
\]
such that, for every $n\ge1$ and $s\ge0$, as $c\to\infty$,
\begin{equation}\label{eq:4.6}
 \norm[X_s\times\R^6]{
 c^{2n}\left(
 z_c-z_\infty-\sum_{j=1}^{n-1}\frac{z_j}{c^{2j}}
 \right)-z_n}
\lesssim \frac{1}{c^2},
\end{equation}
and
\[
\left\| e_c-mc^2
 -e_\infty-\sum_{j=1}^{n}\frac{\mathfrak e_j}{c^{2j}}
 \right\|\lesssim \frac{1}{c^{2n+2}}.
\]
\end{lemma}
\begin{proof}
The coefficient $z_1$ is given by Lemma~\ref{lem:4.1}.
For $n\ge2$, define $z_n$ recursively by \eqref{eq:1.13}.
Its right-hand side depends only on $z_1,\ldots,z_{n-1}$ and,
by Lemmas~\ref{lem:2.2} and~\ref{lem:2.9}, belongs to
\[
 \bigcap_{s\ge2}
 \left(
 \cW^{s-1}_{\mathrm{up}}\times
 \cW^{s-1}_{\mathrm{down}}\times\R\times\{0\}
 \right).
\]
The isomorphism in Lemma~\ref{lem:2.9} therefore gives
\[
 z_n=(f_n,h_n,\lambda_n,0)
 \in\bigcap_{s\ge0}
 \left(
 \cW^s_{\mathrm{up}}\times
 \cW^s_{\mathrm{down}}\times\R\times\{0\}
 \right).
\]

Fix $s\ge2$ and set
\[
 \begin{gathered}
 z_{c,n}:=
 c^{2n}\left(
 z_c-z_\infty-\sum_{j=1}^{n-1}c^{-2j}z_j
 \right),\\
 \zeta_{c,n}:=
 z_\infty+\sum_{j=1}^{n-1}c^{-2j}z_j,
 \qquad
 z_c=\zeta_{c,n}+c^{-2n}z_{c,n}.
 \end{gathered}
\]
For $n=1$, Lemma~\ref{lem:4.1} gives
\[
 \norm[X_s\times\R^6]{z_{c,1}-z_1}\le C_sc^{-2}.
\]
Suppose \eqref{eq:4.6} holds at order $n-1$, where $n\ge2$.
Then
\[
 \begin{aligned}
 z_{c,n}
 &=c^2(z_{c,n-1}-z_{n-1}),\\
 \norm[X_s\times\R^6]{z_{c,n}}
 &\le C_{n-1,s},\\
 \norm[X_s\times\R^6]{\zeta_{c,n}-z_\infty}
 &\le\sum_{j=1}^{n-1}c^{-2j}
       \norm[X_s\times\R^6]{z_j}
 \le C_{n,s}c^{-2}.
 \end{aligned}
\]

Using \eqref{eq:1.13} at orders $1,\ldots,n$, the polynomial
identity \eqref{eq:4.1} yields
\[
 \begin{aligned}
 &c^{2n}\widehat{\cF}_c(\zeta_{c,n})
       +\widehat{\cL}_\infty z_n\\
 ={}&
 \sum_{k=2}^3\frac1{k!}
 \sum_{\substack{1\le i_1,\ldots,i_k\le n-1\\
                  i_1+\cdots+i_k\ge n+1}}
 c^{-2(i_1+\cdots+i_k-n)}
 D_z^k\widehat{\cF}_\infty(z_\infty)
 [z_{i_1},\ldots,z_{i_k}]\\
 &+
 \sum_{\ell=1}^2\sum_{k=1}^3\frac1{k!}
 \sum_{\substack{1\le i_1,\ldots,i_k\le n-1\\
                  \ell+i_1+\cdots+i_k\ge n+1}}
 c^{-2(\ell+i_1+\cdots+i_k-n)}
 D_z^k\cC_\ell(z_\infty)
 [z_{i_1},\ldots,z_{i_k}].
 \end{aligned}
\]
All sums are finite, and Lemma~\ref{lem:2.2} gives
\[
 \norm[X_{s-1}\times\R^6]{
 c^{2n}\widehat{\cF}_c(\zeta_{c,n})
       +\widehat{\cL}_\infty z_n}
 \le C_{n,s}c^{-2}.
\]

Since $\widehat{\cF}_c(z_c)=0$ and
$D_z^4\widehat{\cF}_c=0$, Taylor expansion at $\zeta_{c,n}$
gives
\[
 \begin{aligned}
 \widehat{\cL}_\infty(z_{c,n}-z_n)
 ={}&
 -\left(
 c^{2n}\widehat{\cF}_c(\zeta_{c,n})
 +\widehat{\cL}_\infty z_n
 \right)\\
 &-\left(
 D_z\widehat{\cF}_c(\zeta_{c,n})
 -\widehat{\cL}_\infty
 \right)z_{c,n}\\
 &-\frac12c^{-2n}
 D_z^2\widehat{\cF}_c(\zeta_{c,n})
 [z_{c,n},z_{c,n}]\\
 &-\frac16c^{-4n}
 D_z^3\widehat{\cF}_c(\zeta_{c,n})
 [z_{c,n},z_{c,n},z_{c,n}].
 \end{aligned}
\]
By \eqref{eq:2.30},
\[
 \begin{aligned}
 &\norm[\cB(X_s\times\R^6,X_{s-1}\times\R^6)]{
 D_z\widehat{\cF}_c(\zeta_{c,n})
 -\widehat{\cL}_\infty}\\
 &\qquad\le
 C_s\left(
 c^{-2}+\norm[X_s\times\R^6]{\zeta_{c,n}-z_\infty}
 \right)
 \le C_{n,s}c^{-2}.
 \end{aligned}
\]
Moreover, for $k=2,3$,
\[
 \norm[X_{s-1}\times\R^6]{
 D_z^k\widehat{\cF}_c(\zeta_{c,n})
 [z_{c,n},\ldots,z_{c,n}]}
 \le C_{n,s}\norm[X_s\times\R^6]{z_{c,n}}^k
 \le C_{n,s}.
\]
Consequently,
\[
 \begin{aligned}
 \norm[X_{s-1}\times\R^6]{
 \widehat{\cL}_\infty(z_{c,n}-z_n)}
 &\le C_{n,s}
 \left(c^{-2}+c^{-2n}+c^{-4n}\right)\\
 &\le C_{n,s}c^{-2}.
 \end{aligned}
\]
Applying Lemma~\ref{lem:2.9}, we obtain
\[
 \begin{aligned}
 \norm[X_s\times\R^6]{z_{c,n}-z_n}
 &\le C_s
 \norm[X_{s-1}\times\R^6]{
 \widehat{\cL}_\infty(z_{c,n}-z_n)}\\
 &\le C_{n,s}c^{-2}.
 \end{aligned}
\]
This proves \eqref{eq:4.6} by induction. In particular,
for every $n\ge1$,
\[
 \begin{aligned}
 &\norm[X_s\times\R^6]{
 z_c-z_\infty-\sum_{j=1}^{n}c^{-2j}z_j}\\
 &\qquad=
 c^{-2n}\norm[X_s\times\R^6]{z_{c,n}-z_n}
 \lesssim \frac{1}{c^{2n+2}}.
 \end{aligned}
\]

Set $h_0:=h_\infty$ and
\[
 \mathfrak e_n:=
 -2m\sum_{j=0}^{n}\ipR{h_j}{h_{n-j}}
 =
 -2m\left(
 2\ipR{h_\infty}{h_n}
 +\sum_{j=1}^{n-1}\ipR{h_j}{h_{n-j}}
 \right),
 \qquad n\ge1.
\]
Lemma~\ref{lem:3.3} gives
\[
 \begin{aligned}
 e_c-mc^2
 &=mc^2\left(
 \norm[L^2]{f_c}^2-\norm[L^2]{g_c}^2-1
 \right)
 =-2m\norm[L^2]{h_c}^2,\\
 e_\infty&=-2m\norm[L^2]{h_\infty}^2.
 \end{aligned}
\]
For every $n\ge1$,
\[
 \begin{aligned}
 &\left|
 e_c-mc^2
 +2m\norm[L^2]{\sum_{j=0}^{n}c^{-2j}h_j}^2
 \right|\\
 &\quad\le
 2m\left(
 \norm[L^2]{h_c}
 +\norm[L^2]{\sum_{j=0}^{n}c^{-2j}h_j}
 \right)
 \norm[L^2]{h_c-\sum_{j=0}^{n}c^{-2j}h_j}\\
 &\quad\lesssim \frac{1}{c^{2n+2}}.
 \end{aligned}
\]
Hence
\[
 \begin{aligned}
 e_c-mc^2
 &=-2m\sum_{j,k=0}^{n}
 c^{-2(j+k)}\ipR{h_j}{h_k}
 +O(c^{-2n-2})\\
 &=e_\infty+\sum_{j=1}^{n}c^{-2j}\mathfrak e_j
 +O(c^{-2n-2}).
 \end{aligned}
\]

\end{proof}

\begin{lemma}\label{lem:4.3}
There are $\rho,r>0$ and a unique holomorphic map
\[
 z:\{|\varepsilon|<\rho\}\longrightarrow X_2^\C\times\C^6
\]
with
\begin{equation}
 \begin{gathered}
 \widehat{\cF}^\C(\varepsilon,z(\varepsilon))=0,
 \qquad z(0)=z_\infty,
 \qquad \norm[X_2^\C\times\C^6]{z(\varepsilon)-z_\infty}<r,\\
 z(\overline\varepsilon)=\overline{z(\varepsilon)},
 \qquad z(c^{-2})=z_c\quad\text{for sufficiently large }c,\\
 \frac{\partial_\varepsilon^nz(0)}{n!}=z_n\quad(n\ge1).
 \end{gathered}
 \label{eq:4.7}
\end{equation}
\end{lemma}

\begin{proof}
By Lemma~\ref{lem:2.2}, \eqref{eq:4.1} complexifies to an
entire polynomial
\[
 \widehat{\cF}^\C:
 \C\times(X_2^\C\times\C^6)\longrightarrow X_1^\C\times\C^6,
 \qquad
 D_z\widehat{\cF}^\C(0,z_\infty)=\widehat{\cL}_\infty^\C.
\]
Lemma~\ref{lem:2.9} and the holomorphic implicit function theorem
\cite{Deimling1985} give the local branch. Since
\[
 \overline{\widehat{\cF}^\C(\varepsilon,z)}
 =\widehat{\cF}^\C(\overline\varepsilon,\overline z),
\]
local uniqueness gives the reality identity. Lemma~\ref{lem:2.10}
and Proposition~\ref{prop:2.11} identify $z(c^{-2})=z_c$.
Taylor expansion of the polynomial equation gives \eqref{eq:1.13};
uniqueness in Lemma~\ref{lem:4.2} identifies every Taylor
coefficient with $z_n$. 
\end{proof}

\begin{lemma}\label{lem:4.4}
Let $s\ge\tfrac12$. If
$z=(f,h,\lambda,0)$ is holomorphic from $\{|\varepsilon|<r\}$
into $X_s^\C\times\C^6$ and satisfies
$\widehat{\cF}^\C(\varepsilon,z(\varepsilon))=0$, then it is
holomorphic into $X_{s+1}^\C\times\C^6$ on the same disc.
\end{lemma}

\begin{proof}
The first two equations give
\begin{align}
 \cD^\C h={}&-\lambda f+
 \bigl[(|x|^{-1}\ast|f|^2)f\bigr]^\C
 +\varepsilon\bigl[(|x|^{-1}\ast|h|^2)f\bigr]^\C,
 \notag\\
 \cD^\C f={}&2mh-\varepsilon\lambda h
 +\varepsilon\bigl[(|x|^{-1}\ast|f|^2)h\bigr]^\C
 +\varepsilon^2\bigl[(|x|^{-1}\ast|h|^2)h\bigr]^\C.
 \label{eq:4.8}
\end{align}
The brackets denote complexifications of real polynomial maps. By
Lemma~\ref{lem:2.2}, both right-hand sides are
$(H^s)^\C$-valued holomorphic maps. Since
$(\cD^\C)^2=-\Delta$,
\begin{equation}
 \begin{gathered}
 p=(1-\Delta)^{-1}p
       +(1-\Delta)^{-1}\cD^\C(\cD^\C p),\\
 (1-\Delta)^{-1}:(H^s)^\C\longrightarrow(H^{s+2})^\C,
 \qquad
 (1-\Delta)^{-1}\cD^\C:(H^s)^\C
                         \longrightarrow(H^{s+1})^\C.
 \end{gathered}
 \label{eq:4.9}
\end{equation}
Applying these bounded linear maps to $p=f,h$ proves the claim.
\end{proof}

\begin{lemma}\label{lem:4.5}
The Taylor series in Lemma~\ref{lem:4.3} has the same radius
$R\in(0,\infty]$ in $X_s^\C\times\C^6$ for every $s\ge0$.
For $0<r<R$, $|\varepsilon|<r$ and $N\ge0$,
\begin{equation}
 \begin{aligned}
 &\norm[X_s^\C\times\C^6]{
 z(\varepsilon)-z_\infty-\sum_{n=1}^N\varepsilon^nz_n}\\
 &\qquad\le
 \frac{(|\varepsilon|/r)^{N+1}}{1-|\varepsilon|/r}
 \sup_{|\zeta|=r}\norm[X_s^\C\times\C^6]{z(\zeta)-z_\infty}.
 \end{aligned}
 \label{eq:4.10}
\end{equation}
\end{lemma}

\begin{proof}
Set
\begin{equation}
 R_s^{-1}:=\limsup_{n\to\infty}
             \norm[X_s\times\R^6]{z_n}^{1/n},
 \qquad R:=R_2\ge\rho>0.
 \label{eq:4.11}
\end{equation}
On $|\varepsilon|<R_2$, the $X_2^\C$-valued Taylor sum solves
the complexified equation by the identity theorem. Iterating
Lemma~\ref{lem:4.4} and using embeddings gives
holomorphy into every $X_s^\C\times\C^6$ on the same disc;
in particular, $R_s\ge R_2>0$ for all $s\ge0$.
For $s\ge\tfrac12$, the Taylor sum on $|\varepsilon|<R_s$
again satisfies the complexified equation. Lemma~\ref{lem:4.4}, iterated,
and Sobolev embeddings give
\[
 t>s\ge\tfrac12
 \quad\Longrightarrow\quad
 R_t\ge R_s\quad\text{and}\quad R_t\le R_s.
\]
Thus $R_s=R$ for $s\ge\tfrac12$.

For $0\le s<\tfrac12$, embedding gives $R_s\ge R$.
If $R=\infty$, equality follows. If $R<\infty$, interpolation gives
\begin{equation}
 \norm[X_{1/2}\times\R^6]{z_n}
 \le C_s\norm[X_s\times\R^6]{z_n}^{\frac{3}{4-2s}}
        \norm[X_2\times\R^6]{z_n}^{\frac{1-2s}{4-2s}}.
 \label{eq:4.12}
\end{equation}
Consequently,
\[
 R^{-1}\le R_s^{-\frac{3}{4-2s}}
                     R^{-\frac{1-2s}{4-2s}},
 \qquad R_s\le R.
\]
This proves $R=R_s$ for all $s\ge0$.

The Banach-valued Cauchy formula gives, for $n\ge1$,
\begin{equation}
 \begin{gathered}
 z_n=\frac1{2\pi\ii}\int_{|\zeta|=r}
            \frac{z(\zeta)-z_\infty}{\zeta^{n+1}}\,\dd\zeta,\\
 \norm[X_s\times\R^6]{z_n}
 \le r^{-n}\sup_{|\zeta|=r}
                \norm[X_s^\C\times\C^6]{z(\zeta)-z_\infty}.
 \end{gathered}
 \label{eq:4.13}
\end{equation}
Summing the geometric majorant proves absolute convergence and
\eqref{eq:4.10}. \eqref{eq:1.16} follows from
$z(c^{-2})=z_c$ for sufficiently large $c$.
\end{proof}

\begin{lemma}\label{lem:4.6}
The function
\[
 \mathfrak e(\varepsilon)
 :=-2m\,\ipR{h(\varepsilon)}{h(\varepsilon)}^\C
\]
is holomorphic on $|\varepsilon|<R$, and
\begin{equation}
 \begin{gathered}
 \mathfrak e(\varepsilon)
       =e_\infty+\sum_{n\ge1}\varepsilon^n\mathfrak e_n,
 \qquad \mathfrak e(c^{-2})=e_c-mc^2
       \quad\text{for sufficiently large }c,\\
 2\varepsilon\mathfrak e'(\varepsilon)
          +3\mathfrak e(\varepsilon)+\lambda(\varepsilon)=0.
 \end{gathered}
 \label{eq:4.14}
\end{equation}

\end{lemma}

\begin{proof}
The complexified inner product is a continuous bilinear form.
Lemmas~\ref{lem:4.2} and~\ref{lem:4.5} give
\[
 \begin{aligned}
 \sum_{n\ge1}|\varepsilon|^n|\mathfrak e_n|
 &\le2m\left[
 \left(\norm[L^2]{h_\infty}
       +\sum_{n\ge1}|\varepsilon|^n\norm[L^2]{h_n}\right)^2
                    -\norm[L^2]{h_\infty}^2\right]<\infty,\\
 \mathfrak e(\varepsilon)
 &=e_\infty+\sum_{n\ge1}\varepsilon^n\mathfrak e_n.
 \end{aligned}
\]
Lemma~\ref{lem:3.3} implies
$\mathfrak e(c^{-2})=-2m\norm[L^2]{h_c}^2=e_c-mc^2$.
For large $c$, $u_c$ is $C^1$ in $H^2$ and
$\ipR{u_c}{\partial_cu_c}=0$. Therefore
\begin{equation}
 \begin{aligned}
 \frac{\dd e_c}{\dd c}
 &=\partial_c\cE_c(u_c)+\mathrm d\cE_c(u_c)[\partial_cu_c]\\
 &=\ipR{(-\ii\alpha\cdot\nabla+2mc\beta)u_c}{u_c}
                    +2\omega_c\ipR{u_c}{\partial_cu_c}\\
 &=\ipR{(-\ii\alpha\cdot\nabla+2mc\beta)u_c}{u_c}.
 \end{aligned}
 \label{eq:4.15}
\end{equation}
The identities in Lemma~\ref{lem:3.3} give
\[
 mc^2\ipR{\beta u_c}{u_c}=e_c,
 \qquad
 c\ipR{(-\ii\alpha\cdot\nabla)u_c}{u_c}=e_c-\omega_c,
\]
and hence
\begin{equation}
 \begin{aligned}
 \frac{\dd e_c}{\dd c}&=\frac{3e_c-\omega_c}{c},\\
 2mc-2c^{-3}\mathfrak e'(c^{-2})
 &=2mc+c^{-1}\bigl(3\mathfrak e(c^{-2})+\lambda(c^{-2})\bigr).
 \end{aligned}
 \label{eq:4.16}
\end{equation}
The differential identity in \eqref{eq:4.14} holds first for
small positive $\varepsilon$ and then throughout the disc by
holomorphy. Comparing coefficients gives
\[
 0=3e_\infty+\lambda_\infty
   +\sum_{n\ge1}\varepsilon^n
           \bigl((2n+3)\mathfrak e_n+\lambda_n\bigr).
\]
This proves the lemma.
\end{proof}

Lemmas~\ref{lem:4.1}--\ref{lem:4.2} and
\ref{lem:4.3}--\ref{lem:4.6} prove Theorem~\ref{thm:1.3}.

\begin{samepage}
\noindent {Pan Chen\\
School of Mathematical Sciences,\\
 Shanghai Jiao Tong University, Shanghai 200240, P.R. China
\\
e-mail: chenpan2020@amss.ac.cn }
\medskip
\\
\noindent {Qi Guo\\
School of Mathematics,\\
Renmin University of China, Beijing, 100872, P.R. China\\
e-mail: qguo@ruc.edu.cn}
\medskip
\\
\noindent {Xiaoyu Zeng \\
Center for Mathematical Sciences Mathematics,\\
University of Technology, Wuhan, 430070, P.R. China\\
\ e-mail: xyzeng@whut.edu.cn}

\end{samepage}

\end{document}